\documentclass{amsart}

\usepackage{pb-diagram}
\usepackage{palatino}
\usepackage{mathpazo}
\usepackage[alphabetic]{amsrefs}

\newcommand{\red}{{\operatorname{red}}}

\renewcommand{\setminus}{\ \rule[2.5pt]{7pt}{1pt}\ }
\DeclareMathOperator{\Alg}{Alg}
\DeclareMathOperator{\image}{image}
\newcommand{\mm}{\mathfrak{m}}
\newcommand{\Spec}{\operatorname{Spec}}
\newcommand{\PP}{\mathbf{P}}
\newcommand{\LL}{\mathcal{L}}
\newcommand{\Aff}{\mathbf{Aff}}
\newcommand{\Stab}{\operatorname{Stab}}

\newcommand{\alg}{\operatorname{alg}}

\newcommand{\A}{\mathcal{A}}
\newcommand{\GG}{\mathcal{G}}

\newcommand{\NN}{\mathcal{N}}
\newcommand{\UU}{\mathcal{U}}
\newcommand{\VV}{\mathcal{V}}
\newcommand{\OO}{\mathcal{O}}

\newcommand{\Ksep}{K_{\operatorname{sep}}}
\newcommand{\Kalg}{K_{\operatorname{alg}}}

\newcommand{\Z}{\mathbf{Z}}
\newcommand{\Q}{\mathbf{Q}}

\newcommand{\tensor}{\otimes}

\newcommand{\reg}{{\operatorname{reg}}}

\newcommand{\Lie}{\operatorname{Lie}}
\newcommand{\GL}{\operatorname{GL}}
\newcommand{\SL}{\operatorname{SL}}
\newcommand{\G}{\mathbf{G}}
\newcommand{\Gm}{\G_m}

\newcommand{\Gal}{\operatorname{Gal}}
\newcommand{\Ga}{\G_a}

\newcommand{\lie}[1]{\mathfrak{#1}}
\newcommand{\glie}{\lie{g}}
\newcommand{\plie}{\lie{p}}
\newcommand{\nlie}{\lie{n}}
\newcommand{\alie}{\lie{a}}

\newcommand{\blie}{\lie{b}}
\newcommand{\ulie}{\lie{u}}

\newcommand{\ad}{\operatorname{ad}}
\newcommand{\Ad}{\operatorname{Ad}}
\newcommand{\Int}{\operatorname{Int}}

\newtheorem*{theorem}{Theorem}

\swapnumbers

\swapnumbers

\newtheorem*{prop}{Proposition}

\swapnumbers

\swapnumbers

\newtheorem{theoremalpha}{Theorem}

\swapnumbers
\newtheorem{stmt}{}[subsection]

\swapnumbers
\theoremstyle{remark}
\newtheorem*{rem}{Remark}
\newtheorem*{example}{Example}

\newtheorem*{defin}{Definition}

\numberwithin{equation}{subsection}

\begin{document}

\author{George J. McNinch}
\address{Department of Mathematics,
         Tufts University,
         503 Boston Avenue,
         Medford, MA 02155,
         USA}
\email{george.mcninch@tufts.edu, mcninchg@member.ams.org}
\thanks{\noindent Research of McNinch supported
  in part by the US NSA award H98230-08-1-0110.}
\author{Donna M. Testerman}
\address{Institut de g\'eom\'etrie, alg\`ebre et topologie,
  B\^atiment BCH,
  \'Ecole Polytechnique F\'ed\'erale de Lausanne,
  CH-1015 Lausanne,
  Switzerland}
\thanks{Research of Testerman supported in part by the Swiss National
  Science Foundation grant PP002-68710.}
\email{donna.testerman@epfl.ch} 
\title{Nilpotent centralizers and Springer isomorphisms}
\date{November 28, 2008}

\begin{abstract}
  Let $G$ be a semisimple algebraic group over a field $K$ whose
  characteristic is very good for $G$, and let $\sigma$ be any
  $G$-equivariant isomorphism from the nilpotent variety to the
  unipotent variety; the map $\sigma$ is known as a Springer
  isomorphism.  Let $y \in G(K)$, let $Y \in \Lie(G)(K)$, and write
  $C_y = C_G(y)$ and $C_Y= C_G(Y)$ for the centralizers. We show that
  the center of $C_y$ and the center of $C_Y$ are smooth group schemes
  over $K$. The existence of a Springer isomorphism is used to treat
  the crucial cases where $y$ is unipotent and where $Y$ is nilpotent.

  Now suppose $G$ to be quasisplit, and write $C$ for the centralizer
  of a rational \emph{regular} nilpotent element.  We obtain a description of
  the normalizer $N_G(C)$ of $C$, and we show that the automorphism of
  $\Lie(C)$ determined by the differential of $\sigma$ at zero is a
  scalar multiple of the identity; these results verify observations
  of J-P. Serre.
\end{abstract}

\maketitle
\setcounter{tocdepth}{1}
\tableofcontents

\section{Introduction}

Let $G$ be a reductive group over the field $K$ and suppose $G$ to be
$D$-standard; this condition means that $G$ satisfies some
\emph{standard hypotheses} which will be described in
\S\ref{sub:d-standard}. For now, note that a semisimple group $G$ is
$D$-standard if and only if the characteristic of $K$ is \emph{very
  good} for $G$.

Consider the closed subvariety $\NN$ of nilpotent elements of the Lie
algebra $\glie = \Lie(G)$ of $G$, and the closed subvariety $\UU$ of
unipotent elements of $G$. Since $G$ is $D$-standard, one may follow
the argument given by Springer and Steinberg
\cite{springer-steinberg}*{3.12} to find a $G$-equivariant isomorphism
of varieties $\sigma:\NN \to \UU$.  The mapping $\sigma$ is called a
\emph{Springer isomorphism}. There are many such maps: the Springer
isomorphisms can be viewed as the points of an affine variety whose
dimension is equal to the semisimple rank of $G$; see the note of
Serre found in \cite{mcninch-optimal}*{Appendix} which shows that
despite the abundance of such maps, each Springer isomorphism induces
the same bijection between the (finite) sets of $G$-orbits in $\NN$
and in $\UU$.  For some more details, see \S
\ref{sub:springer-iso-exists} below.

Let $y \in G(K)$ and $Y \in \glie(K)$. Since $G$ is $D$-standard, we
observe in \ref{stmt:smooth-centralizers} -- following Springer and
Steinberg \cite{springer-steinberg} -- that the centralizers $C_G(y)$
and $C_G(Y)$ are smooth group schemes over $K$.  The first main
result of this paper is as follows: 
\begin{theoremalpha}
  \label{smoothness-theorem}
  Let $Z_y = Z(C_G(y))$ and $Z_Y = Z(C_G(Y))$ be the centers of the
  centralizers. 
  \begin{itemize}
  \item[(a)] $Z_y$ and $Z_Y$ are smooth group schemes over $K$.
  \item[(b)] $Y \in \Lie(Z_Y)$.
  \end{itemize}
\end{theoremalpha}

See \S\ref{sub:center-as-scheme} for more details regarding the
subgroup schemes $Z_y \subset C_G(y)$ and $Z_Y \subset C_G(Y)$. The
existence of a Springer isomorphism plays a crucial role in the proof of
Theorem \ref{smoothness-theorem}.

Keep the assumptions on $G$, and suppose in addition that $G$ is
\emph{quasisplit} over $K$; under these assumptions, one can find a
$K$-rational regular nilpotent element $X \in \glie(K)$
\cite{mcninch-optimal}*{Theorem 54}. Write $C = C_G(X)$ for the
centralizer of $X$; it is a smooth group scheme over $K$
\ref{stmt:smooth-centralizers}.  

Our next result concerns the normalizer of $C$ in $G$; write $N =
N_G(C)$.
 \begin{theoremalpha}
   \label{N-mod-C-description}
   \begin{enumerate}
   \item[(i)] $N$ is smooth over $K$ and is a solvable group.
   \item[(ii)] If $r$ denotes the semisimple rank of $G$, then $\dim N
     = 2r + \dim \zeta_G$, where $\zeta_G$ denotes the center of $G$.
   \item[(iii)] There is a 1 dimensional torus $S \subset N$ which is
     not central in $G$ such that $S \cdot \zeta_G^o$ is a maximal
     torus of $N$.
   \end{enumerate}
 \end{theoremalpha}

Fix now a cocharacter $\phi$ associated with the nilpotent element $X$;
cf.   \ref{stmt:associated-cochar}.
\begin{theoremalpha}
  \label{N-mod-C-weights}
  Assume that the derived group of $G$ is quasi-simple. Then the Lie
  algebra of $N/C$ decomposes as the direct sum
  \begin{equation*}
    \Lie(N/C) =\Lie(S_0) \oplus
  \bigoplus_{i=2}^r \Lie(N/C)(\phi;2k_i - 2),
  \end{equation*}
  where $k_1 \le k_2 \le k_3 \le \dots \le k_r$ are the \emph{exponents} of the Weyl group
  of $G$, and where $S_0$ is the image of $S$ in $N/C$.
\end{theoremalpha}

We will deduce several consequences from Theorems
\ref{N-mod-C-description} and \ref{N-mod-C-weights}. First,
\begin{theoremalpha}
  \label{theorem:RuN-split}
  The unipotent radical of $N_{/\Kalg}$ arises by base change from a
  split unipotent $K$-subgroup of $N$.
\end{theoremalpha}
In older language, Theorem \ref{theorem:RuN-split} asserts that the
unipotent radical of $N$ is \emph{defined and split over $K$}.  Next,
fix a Springer isomorphism $\sigma$ and write $u = \sigma(X)$.  The
unipotent radical of the group $C$ is defined over $K$, and $C$ is the
product of $R_u(C)$ with the center $\zeta_G$ of $G$; see
\ref{stmt:centralizer-description}.  The restriction of $\sigma$ to
$R_u(C)$ yields an isomorphism of varieties
\begin{equation*}
  \gamma = \sigma_{\mid \Lie(R_uC)}:\Lie(R_uC) \xrightarrow{\sim} R_uC
\end{equation*}
satisfying $\gamma(0) = \sigma_{\mid \Lie(R_uC)}(0) = 1$. So the
tangent mapping $d\gamma_0$ yields a linear automorphism of the
tangent space
\begin{equation*}
  \Lie(R_uC) = T_1(R_uC).
\end{equation*}
\begin{theoremalpha}
  \label{dsigma-multiple-of-identity}
  Suppose that the derived group of $G$ is quasi-simple.
  \begin{enumerate}
  \item[(1)] The mapping $(d\gamma)_0$ is a scalar multiple of the
    identity automorphism of $\Lie(R_uC)$.
  \item[(2)] Let $B$ a Borel subgroup of $G$ with unipotent radical $U$.
    Then   $\sigma_{\mid \Lie U}:\Lie U \to U$ is an isomorphism, and
    $d(\sigma_{\mid \Lie U})_0:\Lie U \to \Lie U$ is a scalar multiple of
    the identity.
  \end{enumerate}
\end{theoremalpha}
We remark that Theorems \ref{N-mod-C-description},
\ref{N-mod-C-weights}, and \ref{dsigma-multiple-of-identity} confirm
the observations made by Serre at the end of
\cite{mcninch-optimal}*{Appendix}.

The paper is organized as follows. In \S\ref{section:gp schemes} we
recall some generalities about group schemes and smoothness; in
particular, we describe conditions under which the center of a smooth
group scheme is itself smooth.  In \S\ref{section:reductive} we recall
some facts about reductive groups that we require; in particular, we
define $D$-standard groups and we recall that element centralizers
in $D$-standard groups are well-behaved.  In \S
\ref{section:center-of-nil-centralizer} we give the proof of Theorem
\ref{smoothness-theorem}.  Finally, \S\ref{section:regular} contains
the proofs of Theorems \ref{N-mod-C-description},
\ref{N-mod-C-weights}, \ref{theorem:RuN-split} and
\ref{dsigma-multiple-of-identity}.

\section{Recollections: group schemes}
\label{section:gp schemes}

The main objects of study in this paper are group schemes over a field
$K$.  For the most part, we restrict our attention to \emph{affine}
group schemes $A$ of finite type over $K$.  We begin with some general
definitions.

\subsection{Basic Definitions}
We collect here some basic notions and definitions concerning group schemes;
for a full treatment, the reader is referred to \cite{DG} or to 
\cite{JRAG}*{part I}.

For a commutative ring $\Lambda$, let us write $\Alg_\Lambda$ for the
category of ``all'' commutative $\Lambda$-algebras \footnote{Taken in some
  universe, to avoid logical problems.}.  We will write $\Lambda' \in
\Alg_\Lambda$ to mean that $\Lambda'$ is an object of this category -- i.e.
that $\Lambda'$ is a commutative $\Lambda$-algebra.

We are going to consider affine schemes over $\Lambda$; an affine
scheme $X$ is determined by a commutative $\Lambda$-algebra $R$: the
algebra $R$ determines a functor $X:\Alg_\Lambda \to \operatorname{Sets}$ by
the rule
\begin{equation*}
  X(\Lambda') =  \operatorname{Hom}_{\Lambda-\alg}(R,\Lambda').
\end{equation*}
The scheme $X$ ``is'' this functor, and one says that $X$ is
represented by the algebra $R$. One usually writes $R=\Lambda[X]$ and
one says that $\Lambda[X]$ is the coordinate ring of $X$.  The affine
scheme $X$ has finite type over $\Lambda$ provided that $\Lambda[X]$
is a finitely generated $\Lambda$-algebra.

A group valued functor $A$ on $\Alg_\Lambda$ which is an affine scheme
will be called an affine group scheme. If $A$ is an affine group
scheme, then $\Lambda[A]$ has the structure of a Hopf algebra over
$\Lambda$.

If $\Lambda' \in \Alg_\Lambda$, we write $A_{/\Lambda'}$ for the group
scheme over $\Lambda'$ obtained by base change. Thus $A_{/\Lambda'}$ is
the group scheme over $\Lambda'$ represented by the $\Lambda'$-algebra
$\Lambda[A] \tensor_\Lambda \Lambda'$.

Let us fix an affine group scheme $A$ of finite type over the field
$K$. Write $K[A]$ for the coordinate algebra of $K$, and 
choose an algebraic closure $\Kalg$ of $K$.

\subsection{Comparison with algebraic groups}

In many cases, the group schemes we consider may be identified with a
corresponding algebraic group; we now describe this identification.

If the algebra $K[A]$ is \emph{geometrically reduced} -- i.e.  is such
that $\Kalg[A] = K[A] \tensor_K \Kalg$ has no non-zero nilpotent
elements -- then also $K[A]$ is reduced.  The $\Kalg$-points
$A(\Kalg)$ of $A$ may be viewed as an affine variety over $\Kalg$;
since it is reduced, $\Kalg[A]$ is the algebra of regular functions on
$A(\Kalg)$.  Moreover, $A(\Kalg)$ together with the $K$-algebra $K[A]$
of regular functions on $A(\Kalg)$ may be viewed as a variety defined
over $K$ in the sense of \cite{borel-LAG} or \cite{springer-LAG}.

Conversely, an algebraic group $B$ defined over $K$
in the sense of \cite{borel-LAG} or \cite{springer-LAG} comes
equipped with a $K$-algebra $K[B]$ for which
$\Kalg[B] = K[B] \tensor_K \Kalg$ is the algebra
of regular functions on $B$. The Hopf algebra $K[B]$
represents a group scheme.

The constructions in the preceding paragraphs are inverse to one
another, and these constructions permit us to identify the category of
linear algebraic groups defined over $K$ with the full subcategory of
the category of affine group schemes of finite type over $K$
consisting of those group schemes with geometrically reduced
coordinate algebras.

There are interesting group schemes in characteristic $p>0$ whose
coordinate algebras are not reduced.
Standard examples of non-reduced group schemes
include the group scheme $\mu_p$ represented by $K[T]/(T^p -1)$ with
co-multiplication given by $\Delta(T) = T \tensor T$, and the group
scheme $\alpha_p$ represented by $K[T]/(T^p)$ with co-multiplication
given by $\Delta(T) = T \tensor 1 + 1 \tensor T$. Note that
$\mu_p$ is a subgroup scheme of the multiplicative group $\Gm$,
and $\alpha_p$ is a subgroup scheme of the additive group $\Ga$.

\subsection{Smoothness}

For $\Lambda \in \Alg_K$, let $\Lambda[\epsilon]$ denote the algebra
of \emph{dual numbers} over $\Lambda$; thus $\Lambda[\epsilon]$ is a
free $\Lambda$-module of rank 2 with $\Lambda$-basis $\{1,\epsilon\}$, and
$\epsilon^2 = 0$.  If $A$ is a group scheme over $K$, the natural
$\Lambda$-algebra homomorphisms
\begin{equation*}
  \Lambda \hookrightarrow \Lambda[\epsilon] \xrightarrow{\pi}  \Lambda
\end{equation*}
yield  corresponding  group homomorphisms
\begin{equation*}
  A(\Lambda) \hookrightarrow A(\Lambda[\epsilon])  \xrightarrow{A(\pi)} A(\Lambda).
\end{equation*}

The Lie algebra $\Lie(A)$ of $A$ is the group functor on $\Alg_K$
given for $\Lambda \in \Alg_K$ by
\begin{equation*}
  \Lie(A)(\Lambda) = \ker(A(\Lambda[\epsilon])  \xrightarrow{A(\pi)} A(\Lambda)).
\end{equation*}
Abusing notation somewhat, we are going to write also $\Lie(A)$ for
$\Lie(A)(K)$.  We have:
\begin{stmt}[\cite{DG}*{II.4}]
  \label{stmt:Lie algebra}
  \begin{itemize}
  \item[(a)] $\Lie(A)$ has the structure of a $K$-vector space, and
    the mapping $\Lie(A) \to \Lie(A)(\Lambda)$ induces an isomorphism
    \begin{equation*}
      \Lie(A)(\Lambda) \simeq \Lie(A) \tensor_K \Lambda
    \end{equation*}
    for each $\Lambda \in \Alg_K$.
  \item[(b)] For $\Lambda \in \Alg_K$ and $g \in A(\Lambda)$, the inner
    automorphism $\Int(g)$ determines by restriction a
    $\Lambda$-linear automorphism $\Ad(g)$ of $\Lie(A)(\Lambda) \simeq
    \Lie(A) \tensor_K \Lambda$; thus $\Ad:A \to \GL(\Lie(A))$ is a
    homomorphism of group schemes over $K$.
  \end{itemize}
\end{stmt}

\begin{stmt}[\cite{DG}*{II.5.2.1, p. 238} or \cite{KMRT}*{(21.8) and (21.9)}]
  \label{stmt:smoothness}
  One says that the group scheme $A$ is \emph{smooth} over $K$ if any
  of the following equivalent conditions hold:
  \begin{enumerate}
  \item[(a)] $A$ is geometrically reduced -- i.e. $A_{/\Kalg}$ is reduced.
  \item[(b)] the local ring $K[A]_I$ is regular, where $I$ is the maximal
    ideal defining the identity element of $A$.
  \item[(c)] the local ring $K[A]_I$ is regular for each prime ideal $I$ of $K[A]$.
  \item[(d)] $\dim_K \Lie(A) = \dim A$, where $\dim A$ denotes the
    dimension of the scheme $A$, which is equal to the Krull dimension
    of the ring $K[A]$.
  \end{enumerate}
\end{stmt}

If $A$ is a group scheme over $K$, we often abbreviate the phrase ``$A$ is smooth
over $K$'' to ``$A$ is smooth'';

\subsection{Reduced subgroup schemes}

The following result is well known; a proof may be found in
\cite{mcninch-testerman}*{Lemma 3}.
\begin{stmt}
  \label{sttm:reduced}
  If $K$ is perfect, there is a unique smooth subgroup $A_\red \subset
  A$ which has the same underlying topological space as $A$. If $B$ is
  any smooth group scheme over $K$ and $f:B \to A$ is a morphism, then
  $f$ factors in a unique way as a morphism $B \to A_\red$ followed by
  the inclusion $A_\red \to A$.
\end{stmt}

Note that if $K$ is not perfect, the subgroup scheme
$(A_{/\Kalg})_\red$ of $A_{/\Kalg}$ may not arise by base change from
a subgroup scheme over $K$; see \cite{mcninch-testerman}*{Example 4}.

\subsection{Fixed points and the center of a group scheme}
For the remainder of \S \ref{section:gp schemes}, let us fix a group
scheme $A$ which is affine and of finite type over the field $K$.  Let
$V$ denote an affine $K$-scheme (of finite type) on which $A$ acts.
Define a $K$-subfunctor $W$ of $V$ as follows: for each $\Lambda \in
\Alg_K$, let
\begin{equation*}
  W(\Lambda) = \{v \in V(\Lambda) \mid av = v \quad 
  \text{for each  $\Lambda' \in \Alg_\Lambda$ and each 
    $a \in A(\Lambda')$}\}.
\end{equation*}
We write $W = V^A$; it is the functor of \emph{fixed points} for the
action of $A$.

In general one indeed must define the set $W(\Lambda)$ as the fixed
point set of all $a \in A(\Lambda')$ for varying $\Lambda'$: e.g. if
$A$ is infinitesimal, $A(K) = \{1\}$ while $W(K)$ is typically a
proper subset of $V(K)$.

Since $V$ is affine -- hence separated -- and since $K$ is a field so
that $K[A]$ is free over $K$, we have:
\begin{stmt}[\cite{DG}*{II.1 Theorem  3.6} or \cite{JRAG}*{I.2.6(10)}]
  \label{stmt:closed-fixed-points}
  $V^A$ is a closed subscheme of $V$.
\end{stmt}

The following assertion is somewhat related to \cite{JRAG}*{I.2.7 (11) and
  (12)}.
\begin{stmt}
  \label{stmt:alg-closed-fixed-points}
  Suppose in addition that $A$ is smooth over $K$. Then for any commutative
  $K$-algebra $K'$ which is an algebraically closed field, we have
  $V^A(K') = V(K')^{A(K')}$. \footnote{Here $V(K')^{A(K')}$ denotes
the subset of $V(K')$ fixed by each element of the group $A(K')$.}
\end{stmt}

\begin{proof}
  It is immediate from definitions that $V^A(K') \subset
  V(K')^{A(K')}$.  In order to prove the inclusion $V(K')^{A(K')}
  \subset V^A(K')$, we will assume (for notational convenience) that
  $K = K'$ is algebraically closed. Suppose that $v \in V(K)$ and
  that $v$ is fixed by each element of $A(K)$.

  Consider now the morphism $\phi:A \to V$ given for each $\Lambda \in
  \Alg_K$ and each $a \in A(\Lambda)$ by the rule $a \mapsto av$.  The
  result will follow if we argue that $\phi$ is a constant morphism.
  But we know that $\phi:A(K) \to V(K)$ is constant. Since $A$ is a
  reduced scheme, the morphism $\phi$ is determined by its values on
  closed points; since $K$ is algebraically closed, the closed points
  are in bijection with $A(K)$; the fact that $\phi$ is constant now
  follows.
\end{proof}

Consider now the action of $A$ on itself by inner automorphisms.  For
any $\Lambda \in \Alg_K$ and any $a \in A(\Lambda)$, let us write
$\Int(a)$ for the inner automorphism $x \mapsto axa^{-1}$ of the
$\Lambda$-group scheme $A_{/\Lambda}$. The fixed point subscheme for
this action is \emph{by definition} the center $Z$ of $A$; thus we
have the following result (see also \cite{DG}*{II.1.3.9}):
\begin{stmt}
  The center $Z$ is a closed subgroup scheme of $A$. For any
   $\Lambda \in \Alg_K$, we have
  \begin{equation*}
    Z(\Lambda) = \{a \in A(\Lambda) \mid \Int(a) \ \text{is the trivial
automorphism of the group scheme $A_{/\Lambda}$}\}.
  \end{equation*}
\end{stmt}

\subsection{Smoothness of the center}
\label{sub:center-as-scheme}
Write $\alie = \Lie(A)$ for the Lie algebra of $A$. Recall from
  \ref{stmt:Lie algebra} the adjoint action $\Ad$ of $A$ on $\alie$.
\begin{stmt}
  \label{stmt:Lie(Z)-as-fixed-points}
  Regarding $\alie$ as a $K$-scheme, the Lie algebra of $Z$ is the
  fixed point subscheme of $\alie$ for the adjoint action of $A$.
\end{stmt}

\begin{proof}
  Since $Z$ is the fixed point subscheme of $A$ for the action of $A$
  on itself by inner automorphisms, the assertion follows from
  \cite{DG}*{II.4.2.5}.
\end{proof}

In particular, $\Lie(Z)$ identifies with the
$K$-points $\alie^{\Ad(A)}(K)$ of this fixed point functor, and one recovers the fixed
point functor from the $K$-points \cite{JRAG}*{I.2.10(3)}:
\begin{equation*}
  \alie^{\Ad(A)}(\Lambda) = \Lie(Z) \tensor_K \Lambda.
\end{equation*}

\begin{stmt}
  The center $Z$ of $A$ is smooth over $K$ if and only if
  \begin{equation*}
    \dim Z  = \dim_K \alie^{\Ad(A)}(K) = \dim_K \Lie(Z).
  \end{equation*}
\end{stmt}

\begin{proof}
  Immediate from \ref{stmt:smoothness} and the observation
  \ref{stmt:Lie(Z)-as-fixed-points}.
\end{proof}

\begin{example}
  Let $K$ be a perfect field of characteristic $p>0$,
  and let $A$ be the smooth group scheme over $K$ for which
  \begin{equation*}
    A(\Lambda) = \left \{
      \begin{pmatrix}
        t & 0 & 0 \\
        0 & t^{p} & s \\
        0 & 0 & 1
      \end{pmatrix}  \mid t \in \Lambda^\times, s \in \Lambda \right \}
  \end{equation*}
  for each $\Lambda \in \Alg_K$.
  The Lie algebra $\alie$ is spanned as a $K$-vector space by the matrices
  \begin{equation*}
    X =
    \begin{pmatrix}
      1 & 0 & 0 \\
      0 & 0 & 0 \\
      0 & 0 & 0 
    \end{pmatrix} \quad Y =
    \begin{pmatrix}
      0 & 0 & 0 \\
      0 & 0 & 1 \\
      0 & 0 & 0 
    \end{pmatrix}.
  \end{equation*}
  
  Write $Z = Z(A)$ for the center of $A$. Since $K$ is perfect, we may
  form the corresponding reduced subgroup scheme $Z_\red \subset Z$ --
  see e.g.  \cite{mcninch-testerman}*{Lemma 3}; $Z_\red$ is a smooth
  group scheme over $K$.

  We are going to argue that $Z$ is not smooth -- i.e. that $Z \neq
  Z_\red$.  Observe first that $\alie$ is an Abelian Lie algebra; thus
  its center $\mathfrak{z}(\alie)$ is all of $\alie$.

  Now, if $\Kalg$ is an algebraic closure of $K$, it is easy to check
  that the center of the group $A(\Kalg)$ is trivial. It follows that
  the smooth group scheme $Z_\red$ satisfies $Z_\red(\Kalg) = 1$; thus
  $Z_\red$ is trivial and $\Lie(Z_\red) = 0$.

  It is straightforward to verify that the multiples of $X$ are the
  only fixed points of $\alie$ under the adjoint action of $A$. Thus
  $\Lie(Z) = \alie^{\Ad(A)}$ has dimension 1 as a $K$-vector space.
  Since $\dim Z = \dim Z_\red = 0$, it follows that $Z$ is not smooth.

  Note that for this  example, both containments in the following
  sequence are proper:
  \begin{equation*}
    \Lie(Z_\red) \subset \Lie(Z) \subset \lie{z}(\alie).
  \end{equation*}
\end{example}

\subsection{Smoothness of certain fixed point subgroup schemes}
\label{sub:diag-smooth-fixed-points}

Recall that a group scheme $D$ over $K$ is \emph{diagonalizable} if
$K[D]$ is spanned as a linear space by the group of characters
$X^*(D)$.  The group scheme $D$ is of \emph{multiplicative type} if
$D_{/\Kalg}$ is diagonalizable. 

Suppose in this section that $D$ is either a group scheme of multiplicative
type, or that $D$ is an \'etale group scheme over $K$ for which the finite group
$D(\Kalg)$ has order invertible in $K$.

Assume that $D$ acts on the group scheme $A$ by group automorphisms:
for any $\Lambda \in \Alg_K$ and any $x \in D(\Lambda)$, the element
$x$ acts on the group scheme $A_{/\Lambda}$ as a group scheme
automorphism.

The fixed points $A^D$ form a closed subgroup scheme of $A$
\label{stmt:closed-fixed-points}.  Moreover, we have:
\begin{stmt}
  \label{stmt:diag-smooth-fixed-points}
  If $A$ is smooth over $K$, then also the fixed point subgroup scheme $A^D$
  is smooth over $K$.
\end{stmt}

\begin{proof}
  According to the ``Th\'eor\`eme de lissit\'e des centralisateurs''
  \cite{DG}*{II.5.2.8 (p. 240)} the result will follow if we know that
  $H^1(D,\Lie(A)) = 0$.  It suffices to check this condition after
  extending scalars; thus we may and will suppose that $D$ is
  diagonalizable or that $D$ is the constant group scheme determined
  by a finite group whose order is invertible in $K$.
 
  In each case, one knows that the cohomology group $H^n(D,M) $ is $0$
  for \emph{all} $D$-modules $M$ and all $n \ge 1$; for a finite group
  with order invertible in $K$, this vanishing is well-known; for a
  diagonalizable group, see \cite{JRAG}*{I.4.3}.
\end{proof}

\subsection{Possibly disconnected groups}
\label{sub:possibly-disconnected}

Let $G$ be a smooth linear algebraic group over $K$.
\begin{stmt}
  \label{sub:center-in-disconnected}
  Suppose that $1 \to G \to G_1 \to E \to 1$ is an exact sequence,
  where $E$ is finite \'etale and $E(\Kalg)$ has order invertible in
  $K$.  If the center of $G$ is smooth, then the center of $G_1$ is
  smooth.
\end{stmt}

\begin{proof}
  Write $Z$ for the center of $G$, write $Z_1$ for the center of $G_1$.
  Note that $E$ acts naturally on $Z$.

  There is an exact sequence of groups
  \begin{equation*}
    1 \to Z^E \to Z_1 \to H \to 1
  \end{equation*}
  for a subgroup $H \subset E$. Since $Z$ is
  smooth, the smoothness of $Z^E$ follows from
  \ref{stmt:diag-smooth-fixed-points}; since $H$ is smooth, one
  obtains the smoothness of $Z_1$ by applying \cite{KMRT}*{Cor.
    (22.12)}.
\end{proof}

\subsection{Split unipotent radicals}
\label{sub:split-uni-radical}

Fix a smooth group scheme $A$ over $K$.  A smooth group scheme $B$
over $K$ is unipotent if each element of $B(\Kalg)$ is unipotent.
Recall that the unipotent radical of $A_{/\Kalg}$ is the maximal
closed, connected, smooth, normal, unipotent subgroup scheme of
$A_{/\Kalg}$.
\begin{stmt}\cite{springer-LAG}*{Prop. 14.4.5}
  If $K$ is perfect, there is a smooth subgroup scheme $R_uA \subset A$
  such that $R_uA_{/\Kalg}$ is the unipotent radical of $A_{/\Kalg}$.
\end{stmt}

If $K$ is not perfect, then in general $R_uA_{/\Kalg}$ does not arise
by base change from a $K$-subgroup scheme of $A$.  The unipotent group
$B$ is said to be \emph{split} provided that there are closed subgroup
schemes
\begin{equation*}
  1 = B_0 \subset B_1 \subset \cdots \subset B_n = B
\end{equation*}
such that $B_i/B_{i-1} \simeq \Ga$ for $1 \le i \le n$.

\begin{theorem}
  Let $A$ be a connected, solvable, and smooth group scheme over $K$.
  Let $T \subset A$ be a maximal torus, and suppose that $\phi:\Gm \to
  T$ is a cocharacter. Write $S$ for the image of $\phi$. If $\Lie(T)$
  is precisely the set of fixed points $\Lie(A)^S$, and if each
  non-zero weight $\lambda$ of $S$ on $\Lie(A)$ satisfies $\langle
  \lambda,\phi \rangle > 0$, then $R_uA$ is defined over $K$ and is a
  split unipotent group scheme.
\end{theorem}

\begin{proof}
  Write $P = P(\phi)$ for the smooth subgroup scheme of $A$ determined
  by $\phi$ as in \cite{springer-LAG}*{\S13.4}; it is the 
  subgroup \emph{contracted by} the cocharacter $\phi$.  Write $M =
  C_A(S)$; $M$ is connected
  \cite{springer-LAG}*{p. 110} and smooth \cite{DG}*{p. 476, cor. 2.5}.
  There is a smooth, connected, normal, unipotent subgroup scheme
  $U(\phi) \subset P$ for which the product morphism
  \begin{equation*}
     M \times U(\phi) \to P
  \end{equation*}
  is an isomorphism of varieties; \cite{springer-LAG}*{13.4.2}.  Moreover,
  since $\langle \lambda,\phi \rangle > 0$ for each weight of $S$ on
  $\Lie(A)$, it follows that $U(-\phi)$ is trivial. Thus \emph{loc. cit.}
  13.4.4 shows that $A = P$.

  Evidently $T \subset M$. Since $\Lie(T) = \Lie(M)$, it follows that
  $M = T$. It follows that $U(\phi)_{/\Kalg}$ is the unipotent radical
  of $A_{/\Kalg}$ as desired.

  Finally, it follows from \cite{springer-LAG}*{14.4.2} that $U(\phi)$ is a
  $K$-split unipotent group, and the proof is complete.
\end{proof}

\subsection{Torus actions on a projective space}
\label{section:torus-actions}

Let $T$ be a split torus over $K$, and let $V$ be a $T$-representation.  For
$\lambda \in X^*(T)$, let $V_\lambda$ be the corresponding weight
space; thus $T$ acts on $V_\lambda$ through the character $\lambda:T
\to \Gm$.  There are distinct characters $\lambda_1,\dots,\lambda_n
\in X^*(T)$ such that
\begin{equation*}
  V = \bigoplus_{i=1}^n V_{\lambda_i};
\end{equation*}
the $\lambda_i$ are the \emph{weights} of $T$ on $V$.  Let us fix a
vector $0 \ne v \in V_{\lambda_1}$.

Consider now the projective space $\PP(V)$ of lines through the origin
in $V$; for a non-zero vector $w \in V$, write $[w]$ for the
corresponding point of $\PP(V)$.  The linear action of $T$ on $V$
induces in a natural way an action of $T$ on $\PP(V)$.

Since $v$ is a weight vector for $T$, the point $[v] \in \PP(V)(K)$
determined by $v$ is fixed by the action of $T$. Consider the tangent
space $M = T_{[v]}\PP(V)$; since $[v]$ is a fixed point of $T$, the
action of $T$ on $\PP(V)$ determines a linear representation of $T$ on
$M$.

\begin{stmt}
  \label{stmt:T-weights-on-tangent-to-P(V)}
  The non-zero weights of $T$ on $M = T_{[v]}\PP(V)$ are the characters $\lambda_i -
  \lambda_1$ for $1 < i \le n$.  Moreover,
  \begin{equation*}
    \dim M_0 = \dim V_{\lambda_1} - 1 \quad \text{and}\quad 
    \dim M_{\lambda_i - \lambda_1} = \dim V_{\lambda_i}, \quad 1 < i \le n.
  \end{equation*}
\end{stmt}

\begin{proof}
  Choose a basis $S_1,S_2,\dots,S_r$ for the dual space of $V^\vee$
  for which $S_i \in V^\vee_{-\lambda_i}$ for $1 \le i \le r$ -- i.e.
  the vector $S_i$ has weight $-\lambda_i$ for the contragredient
  action of $T$ on $V^\vee$.  Without loss of generality, we may and
  will assume that $S_1$ satisfies $S_1(v) \ne 0$ and that $S_i(v) =
  0$ for $2 \le i \le n$.

  Now consider the affine open subset $\VV = \PP(V)_{S_1}$ of $\PP(V)$
  defined by the non-vanishing of $S_1$.  One knows that $[v]$ is a
  point of $\VV$. Moreover, $\VV \simeq \Aff^{r-1}$ where $r = \dim
  V$. Since $S_1$ is a weight vector for the action of the torus $T$, it
  is clear that $\VV$ is a $T$-stable subvariety of $\PP(V)$. More precisely,
  $\VV$ identifies with the affine scheme
  $\Spec(\A)$ where $\A$ is the $T$-stable subalgebra
  \begin{equation*}
    \A = k\left[\dfrac{S_2}{S_1},\dfrac{S_3}{S_1},\dots,\dfrac{S_r}{S_1} \right]
  \end{equation*}
  of the field of rational functions $k(\PP(V))$.
  
  Under this identification, the point $[v] \in \VV$ corresponds to
  the point $\vec 0$ of $\Aff^{r-1}$; i.e. to the maximal ideal $\mm =
  \left ( \dfrac{S_2}{S_1},\dfrac{S_3}{S_1},\dots,\dfrac{S_r}{S_1} \right )
  \subset \A$. Now, $\mm$ and $\mm^2$ are $T$-invariant; since
  $\dfrac{S_i}{S_1}$ has weight $-\lambda_i + \lambda_1$, evidently the
  weights of $T$ in its representation on $\mm/\mm^2$ are of the form
  $-\lambda_i + \lambda_1$, and one has
  \begin{equation*}
    \dim (\mm/\mm^2)_0 = \dim V_{\lambda_1} - 1 \quad \text{and}\quad 
    \dim (\mm/\mm^2)_{-\lambda_i+\lambda_1} = \dim V_{\lambda_i}, \quad 1 < i \le n.
  \end{equation*}

  The assertion now follows since there is a $T$-equivariant isomorphism
  between the tangent space to $\PP(V)$ at $[v]$ -- i.e.  the space $M
  = T_{[v]} \PP(V)$ -- and the contragredient representation
  $(\mm/\mm^2)^\vee.$
\end{proof}

\subsection{Surjective homomorphisms between 
group schemes; normalizers}

In this section, let us fix group schemes $G_1$ and $G_2$ over $K$,
and suppose that $f:G_1 \to G_2$ is a \emph{surjective} homomorphism
of group schemes; recall that $f$ is surjective provided that the
comorphism $f^*:K[G_2] \to K[G_1]$ is injective (cf.
\cite{KMRT}*{Prop. 22.3}).

The mapping $f$ is said to be \emph{separable} provided that
$df:\Lie(G_1) \to \Lie(G_2)$ is surjective as well.

Let $C_2 \subset G_2$ be a subgroup scheme, and let $C_1 = f^{-1}C_2$
be the scheme-theoretic inverse image.
\begin{stmt}
  \label{stmt:inverse-image}
  \begin{enumerate}
  \item[(a)] The mapping obtained by restriction $f_{\mid C_1}:C_1 \to
    C_2$ is surjective.
  \item[(b)] If $C_1$ is smooth, then $C_2$ is smooth.
  \item[(c)] If $f$ is separable and $C_2$ is smooth, then $C_1$ is
    smooth.
  \item[(d)] Suppose that $f$ is separable, and that either $C_1$ or
    $C_2$ is smooth.  Then both $C_1$ and $C_2$ are smooth, and
    $f_{\mid C_1}$ is separable.
  \end{enumerate}
\end{stmt}

\begin{proof}
  (a) and (b) follow from \cite{KMRT}*{Prop. 22.4}.

  We now prove (c). Since $f$ is separable and surjective,
  \cite{KMRT}*{Prop. 22.13} shows that $\ker f$ is a smooth group
  scheme over $K$. Note that $\ker f \subset C_1$. If $C_2$ is smooth,
  the smoothness of $C_1$ now follows from \cite{KMRT}*{Cor. 22.12}.
  
  We finally prove (d). The smoothness assertions have already been
  proved. We again know $\ker f$ to be  smooth over $K$.
  In particular, $\dim \ker f = \dim \ker df$. Since $\ker f \subset
  C_1$, we have
  \begin{equation*}
    \dim \image(df_{\mid C_1}) = \dim \Lie(C_1) - \dim \ker df_{\mid C_1}
    = \dim C_1 - \dim \ker f_{\mid C_1} = \dim C_2,
  \end{equation*}
  where we have used \cite{KMRT}*{Prop. 22.11} for the final equality;
  since $C_2$ is smooth, it follows that $df_{\mid C_1}:\Lie(C_1) \to
  \Lie(C_2)$ is surjective.
\end{proof}

Write $N_2 = N_{G_2}(C_2)$ for the normalizer of $C_2$ in $G_2$.
Thus $N_2$ is the subgroup functor given for $\Lambda
\in \Alg_K$ by the rule
\begin{align*}
  N_2(\Lambda) =& \{g \in G_2(\Lambda) \mid g \ \text{normalizes the subgroup
    scheme $C_{2/\Lambda} \subset G_{2/\Lambda}$}\} \\
   =& \{ g \in G_2(\Lambda) \mid gC_2(\Lambda')g^{-1} = C_2(\Lambda')
   \ \text{for all}\ \Lambda' \in \Alg_{\Lambda}\}.
\end{align*}
According to \cite{DG}*{II.1 Theorem 3.6(b)}, $N_2$ is a closed subgroup
scheme of $G_2$.

As a consequence of \ref{stmt:inverse-image}, we find the following:
\begin{stmt}
  \label{stmt:smooth-normalizer}
  Set $N_1 = f^{-1}N_2$.
  \begin{enumerate}
  \item[(a)] $N_1 =  N_{G_1}(C_1)$.
  \item[(b)] $f_{\mid N_1}:N_1 \to N_2$ is surjective.
  \item[(c)] If $N_1$ is smooth, then $N_2$ is smooth.
  \item[(d)] If $f$ is separable and $N_2$ is smooth, then $N_1$ is
    smooth.
  \item[(e)] Suppose that $f$ is separable and that either $N_1$ or
    $N_2$ is smooth. Then both $N_1$ and $N_2$ are smooth, and
    $f_{\mid N_1}$ is separable.
  \end{enumerate}
\end{stmt}


\section{Recollections: reductive groups}
\label{section:reductive}

Let $G$ be a connected and reductive group over $K$. Thus $G$ is a smooth group
scheme over $K$, or equivalently $G$ is a linear algebraic group
defined over $K$. To say that $G$ is reductive means that the
unipotent radical of $G_{/\Kalg}$ is trivial. We are going to write
$\zeta_G = Z(G)$ for the center of $G$.

Some results will be seen to hold for a reductive group $G$ in case
$G$ is \emph{$D$-standard}; in the next few sections, we explain this
condition.  We must first recall the notions of good and bad
characteristic.

\subsection{Good and very good primes}
Suppose that $H$ is a smooth group scheme over $K$ -- i.e. an
algebraic group over $K$ -- for which $H_{/\Kalg}$ is quasisimple;
thus $H$ is geometrically quasisimple. Write $R$ for the root system
of $H$. The characteristic $p$ of $K$ is said to be a bad prime for
$R$ -- equivalently, for $H$ -- in the following circumstances: $p=2$
is bad whenever $R \not = A_r$, $p=3$ is bad if $R = G_2,F_4,E_r$, and
$p=5$ is bad if $R=E_8$.  Otherwise, $p$ is good.

A good prime $p$ is \emph{very good} provided that either
$R$ is not of type $A_r$, or that $R=A_r$ and $r \not
\equiv -1 \pmod p$.

If $H$ is any reductive group, one may apply \cite{KMRT}*{Theorems
  26.7 and 26.8} \footnote{\cite{KMRT} only deals with the semisimple
  case; the extension to a general reductive group is not difficult to
  handle, and an argument is sketched in the footnote found in
  \cite{mcninch-testerman}*{\S2.4}.}  to see that there is a possibly
inseparable central isogeny
\begin{equation}
  \label{eq:cover-G}
   \quad R(H) \times \prod_{i=1}^m H_i \to H
\end{equation}
where the radical $R(H)$ of $H$ is a torus, and where for $1 \le i \le
m$ there is an isomorphism $H_i \simeq R_{L_i/K} J_i$ for a finite
separable field extension $L_i/K$ and a geometrically quasisimple,
simply connected group scheme $J_i$ over $L_i$; here, $R_{L_i/K} J_i$
denotes the ``Weil restriction'' -- or restriction of scalars -- of
$J_i$ to $K$, cf.  \cite{springer-LAG}*{\S11.4}.  The $H_i$ are
uniquely determined by $H$ up to order of the factors.  Then $p$ is
\emph{good}, respectively \emph{very good}, for $H$ if and only if
that is so for $J_i$ for every $1 \le i \le m$.

\subsection{$D$-standard}
\label{sub:d-standard}
Recall from \S\ref{sub:diag-smooth-fixed-points} the notion of a
diagonalizable group scheme, and of a group scheme of multiplicative
type.

\begin{stmt}
  If $D$ is subgroup scheme of $G$ of multiplicative type, the
  connected centralizer $C_G(D)^o$ is reductive.
\end{stmt}

When $D$ is smooth, the preceding result is well-known: the group $D$
is the direct product of a torus and a finite \'etale group scheme all
of whose geometric points have order invertible in $K$.  The
centralizer of a torus is (connected and) reductive, and one is left
to apply a result of Steinberg \cite{steinberg-endomorphisms}*{Cor.
  9.3} which asserts that the centralizer of a semisimple automorphism
of a reductive group has reductive identity component.  In fact, the
result remains valid when $D$ is no longer smooth; a proof will appear
elsewhere.

Consider reductive groups $H$ which are direct products
\begin{equation*}
  (*) \quad
  H=H_1  \times T
\end{equation*}
where $T$ is a torus, and where $H_1$ is a semisimple group 
for which the characteristic of $K$ is \emph{very good}.


\begin{defin}
  A reductive group $G$ is \emph{$D$-standard} if there exists a
  reductive group $H$ of the form $(*)$, a subgroup $D \subset H$ such
  that $D$ is of multiplicative type, and a separable isogeny between
  $G$ and the reductive group $C_H(D)^o$. \footnote{This definition
    does not require the knowledge that $C_H(D)^o$ is reductive: if
    there is an isogeny between $G$ and $C_H(D)^o$, then $C_H(D)^o$ is
    reductive.}
\end{defin}

\begin{stmt}[\cite{mcninch-optimal}*{Remark 3}]
  For any $n \ge 1$, the group $\GL_n$ is $D$-standard. The group
  $\SL_n$ is $D$-standard if and only if $p$ does not divide $n$.
\end{stmt}

In order to prove \ref{stmt:D-standard-quasisimple-parts} below, we
first observe:
\begin{stmt}
  \label{stmt:base-change-M}
  Let $M,G_1,G_2$ be affine group schemes of finite type over $K$.
  Let $f:G_1\to G_2$ be a surjective morphism of group schemes,
  suppose that $\ker f$ is central in $G_1$, and let $\phi:M \to G_2$
  be a homomorphism of group schemes for which
  $\phi^{-1}(\zeta_{G_2})$ is central in $M$.
  Consider the group scheme $\widetilde{M}$ defined by the Cartesian diagram:
  \begin{equation*}
    \begin{diagram}
      \node{\widetilde{M} = M \times_{G_2} G_1}  \arrow{s,t}{\widetilde{\phi}} 
      \arrow{e,t}{\tilde f} \node{M} \arrow{s,b}{\phi} \\
      \node{G_1} \arrow{e,t}{f} \node{G_2}
    \end{diagram}    
  \end{equation*}
  Then 
  \begin{enumerate}
  \item[(a)] $\widetilde{\phi}^{-1}(\zeta_{G_1})$ is central in $\widetilde{M}$.
  \item[(b)] Suppose that $G_1,G_2$ are connected and reductive, that $f$ is a
    separable isogeny, and that $M$ is connected and quasisimple. Then
    $\widetilde{M}$ is connected and quasisimple.
  \end{enumerate}
\end{stmt}

\begin{proof}
  To prove (a), let $N = \widetilde{\phi}^{-1}(\zeta_{G_1})$. It is
  enough to show that $\widetilde{\phi}(N)$ is central in $G_1$ and
  that $\widetilde{f}(N)$ is central in $M$. The first of these
  observations is immediate from definitions, while the second follows
  from assumption on the mapping $\phi:M \to G_2$ once we observe that
  $\widetilde{f}(N) \subset \phi^{-1}(\zeta_{G_2}).$

  For (b), we view $\widetilde{f}$ as arising by base change from $f$.
  Then $\widetilde{f}$ is an isogeny since $\ker(f)_{/\Kalg}$ and
  $\ker(\widetilde{f})_{\Kalg}$ coincide. Moreover, it follows from
  \cite{liu}*{Prop 4.3.22} that $\widetilde{f}$ is separable (since it
  is \'etale). Thus $\widetilde{f}$ is a separable isogeny; since
  $\widetilde{M}$ is separably isogenous to a connected quasisimple
  group, it is itself connected and quasisimple.
\end{proof}

\begin{stmt}
  \label{stmt:D-standard-quasisimple-parts}
  Suppose that the $D$-standard reductive group $G$ is split over $K$.
  There are $D$-standard reductive groups $M_1,\dots,M_d$ together
  with a homomorphism $\Phi:M \to G$, where $M = \prod_{i=1}^dM_i$, such
  that the following hold: 
    \begin{itemize}
    \item[(a)] The derived group of $M_i$ is geometrically quasisimple
      for $1 \le i \le d$.
      \item[(b)] $\Phi$ is surjective and separable.
      \item[(c)] For $1 \le i < j \le d$, the image in $G$ of $M_i$ and
      $M_j$ commute.
    \item[(d)] The subgroup scheme $\Phi^{-1}(\zeta_G)$ is central in
      $\prod_{i=1}^d M_i$.
    \end{itemize}
\end{stmt}

\begin{proof}
  We argue first that it suffices to prove the result after replacing
  $G$ be a separably isogenous group. More precisely, we prove: $(*)$
  if $f:G_1 \to G_2$ is a separable isogeny between $D$-standard
  reductive groups $G_1$ and $G_2$, then
  \ref{stmt:D-standard-quasisimple-parts} holds for $G_1$ if and only
  if it holds for $G_2$.

  Suppose first that the conclusion of
  \ref{stmt:D-standard-quasisimple-parts} is valid for $G_1$.  If
  $\Phi:M \to G_1$ is a homomorphism for which (a)--(d) hold, then
  evidently (a)--(d) hold for $f \circ \Phi$.

  Now suppose that the conclusion of
  \ref{stmt:D-standard-quasisimple-parts} is valid for $G_2$, and that
  $\Phi:M \to G_2$ is a homomorphism for which (a)--(d) hold. For each
  $1\le j \le d$ write $\Phi_j$ for the composite of $\Phi$ with the
  inclusion of $M_j$ in the product.  Form the group $\widetilde{M_j}
  = M_j \times_{G_2} G_1$ as in \ref{stmt:base-change-M}.  Then by (b)
  of \emph{loc. cit.}, $\widetilde{M_j}$ is quasisimple.  Moreover,
  \emph{loc. cit.} (a) shows the kernel of $\widetilde{\Phi_j}$ it be
  central in $\widetilde{M_j}$.

  Note that the image of $\widetilde{\Phi_j}$ is mapped to the image
  of $\Phi_j$ by $f$. Now, $f$ is a separable isogeny, hence in
  particular $f$ is central; i.e.  $\ker f$ is central. It follows
  that the image of $\widetilde{\Phi_i}$ commutes with the image of
  $\widetilde{\Phi_j}$ whenever $1 \le i \ne j \le n$.  We can thus
  form the homomorphism $\widetilde{\Phi}:\prod_{j=1}^d
  \widetilde{M_j} \to G_1$ whose restriction to each $\widetilde{M_j}$
  is just $\widetilde{\Phi_j}$, and it is clear that (a)--(d) hold for
  $\widetilde{\Phi}$; this completes the proof of $(*)$.

  In view of the definition of a $D$-standard group, we may now
  suppose that $G$ is the connected centralizer $C_{H_1}(D)^o$ of a diagonalizable
  subgroup scheme $D \subset H_1 = H \times S$, where $H$ is a semisimple
  group in very good characteristic and $S$ a torus.

  We may use \cite{springer-LAG}*{8.1.5} to write $G$ as a commuting
  product of its minimal non-trivial connected, closed, normal
  subgroups $J_i$ for $i = 1,2,\dots,n$. Fix a maximal torus $T \subset
  G$, so that $T_i = (T \cap J_i)^o$ is a maximal torus of $J_i$ for
  each $i$.

  Now set $T^i = \prod_{i \ne j} T_j$; then $T^i$ is a torus in $G$.
  Moreover, $J_i$ is the derived subgroup of the reductive group $M_i
  = C_G(T_i)$.

  Now, $M_i$ is the connected centralizer in $H_1$ of the
  diagonalizable subgroup $\langle T^i,D\rangle$; thus $M_i$ is
  $D$-standard.

  Finally, putting $M = \prod_i M_i$, we have a natural surjective
  mapping $M \to G$ for which (a)-(d) hold, as required.
\end{proof}

\subsection{Existence of Springer Isomorphisms}
\label{sub:springer-iso-exists}

Let $G$ denote a $D$-standard reductive group.  We write $\NN = \NN(G)
\subset \glie$ for the \emph{nilpotent variety} of $G$ and $\UU =
\UU(G) \subset G$ for the \emph{unipotent variety} of $G$.

By a \emph{Springer isomorphism}, we mean a map 
\begin{equation*}
  \sigma:\NN \to \UU
\end{equation*}
which is a $G$-equivariant isomorphism of varieties over $K$.

The first assertion of the following Theorem -- the existence of a
Springer isomorphism -- is due essentially to Springer; see e.g.
\cite{springer-steinberg}*{III.3.12} for the case of an algebraically
closed field, or see \cite{springer-iso}.  The second assertion was
obtained by Serre and appears in the appendix to \cite{mcninch-optimal}.

\begin{theorem}[Springer, Serre]
  \begin{enumerate}
  \item There is a Springer isomorphism $\sigma:\NN \to \UU$.
  \item Any two Springer isomorphisms induce the same mapping between
    the set of $G(\Kalg)$-orbits in $\UU(\Kalg)$ and the set of
    $G(\Kalg)$-orbits in $\NN(\Kalg)$, where $\Kalg$ is an algebraic closure of $K$.
  \end{enumerate}
\end{theorem}

\begin{proof}
  We sketch the argument for assertion (1) in order to point out the
  role of the $D$-standard assumption made on $G$.

  If $G$ is semisimple in very good characteristic, the nilpotent
  variety $\NN$ and the unipotent variety $\UU$ are both normal.
  Indeed, for $\UU$, one knows \cite{springer-steinberg}*{III.2.7}
  that $\UU$ is normal whenever $G$ is simply connected (with no
  condition on $p$). Moreover, one knows that the normality of $\UU$
  is preserved by separable isogeny \footnote{More precisely, if
    $\pi:G \to G_1$ is a separable central isogeny, the restriction of
    $\pi$ determines an isomorphism between $\UU(G)$ and $\UU(G_1)$.}.
  In positive characteristic the normality of $\NN$ for a semisimple
  group $G$ is a result of Veldkamp (for most $p$) and of Demazure
  when the characteristic is very good for $G$; see
  \cite{jantzen-nil}*{8.5}.  Using the normality of $\UU$ and of
  $\NN$, Springer showed that \cite{springer-iso} there is a
  $G$-equivariant isomorphism as required.

  To conclude that assertion (1) is valid for any $D$-standard groups,
  it suffices to observe the following: (i) if $\pi:G \to G_1$ is a separable
  isogeny, then there is a Springer isomorphism for $G$ if and only if
  there is a Springer isomorphism for $G_1$, and (ii) if $H$ is a
  reductive group for which there is a Springer isomorphism, and if $D
  \subset H$ is a subgroup of multiplicative type, then $C_H^o(D)$ has
  a Springer isomorphism.
\end{proof}

We note a related result for certain not-necessarily-connected
reductive groups.
\begin{stmt}
  \label{stmt:diag-centralizer-component-group}
  Let $G$ be a connected reductive group for which there is a Springer
  isomorpism $\sigma:\NN(G) \to \UU(G)$. Let $D \subset G$ be a
  subgroup of multiplicative type, and let $M = C_G(D)$.  
  \begin{enumerate}
  \item[(a)]  $\sigma$ restricts to an isomorphism $\NN(M) \to \UU(M)$.
  \item[(b)] The finite group $M(\Kalg)/M^o(\Kalg)$ has order
    invertible in $K$.
  \end{enumerate}
\end{stmt}

\begin{proof}
  Assertion (a) follows from the observations: $\NN(M) = \NN(G)^D$ and
  $\UU(M) = \UU(G)^D$. To prove (b), note that $\NN(M) = \NN(M^o)$ is
  connected, so that by (a), also $\UU(M)$ is connected. Thus $\UU(M)
  \subset M^o$ and (b) follows at once.
\end{proof}

\subsection{Smoothness of some subgroups of $D$-standard groups}
For any algebraic group, and any element $x \in G$, let $C_G(x)$
denote the centralizer subgroup scheme of $G$.  Then by definition
$\Lie C_G(x) = \lie{c}_\glie(x)$, where $\lie{c}_\glie(x)$ denotes the
centralizer of $x$ in the Lie algebra $\glie$, but since the
centralizer may not reduced, the dimension of $\lie{c}_\glie(x)$ may
be larger than the dimension $\dim C_G(x) = \dim
C_G(x)_{\operatorname{red}}$, where $C_G(x)_{\operatorname{red}}$
denotes the corresponding \emph{reduced} -- hence smooth -- group
scheme. Similar remarks hold when $x \in G$ is replaced by an element
$X \in \glie$.

When $G$ is a $D$-standard reductive group, this difficulty does not
arise. Indeed:
\begin{stmt}
  \label{stmt:smooth-centralizers}
  Let $G$ be $D$-standard, let $x \in G(K)$, and let $X \in \glie =
  \glie(K)$.  Then $C_G(x)$ and $C_G(X)$ are smooth over $K$. In other
  words,
  \begin{equation*}
    \dim C_G(x) = \dim \lie{c}_\glie(x) \quad  \text{and} \quad
    \dim C_G(X) = \dim \lie{c}_\glie(X).
  \end{equation*}
  In particular,
  \begin{equation*}
    \Lie C_G(x)_{\operatorname{red}} = \lie{c}_\glie(x) \quad \text{and} \quad
    \Lie C_G(X)_{\operatorname{red}} = \lie{c}_\glie(X).
  \end{equation*}
\end{stmt}

\begin{proof}
  When $G$ is semisimple in very good characteristic, the result
  follows from \cite{springer-steinberg}*{I.5.2 and I.5.6}. The
  extension to $D$-standard groups is immediate; the verification is
  left to the reader. \footnote{Complete details of the reduction from
    the case of a $D$-standard group to that of a semisimple group in
    very good characteristic can be given along the lines of the
    argument used in the proof of \ref{stmt:N-smooth}.}
\end{proof}

Similar assertions holds for the center of $G$, as follows:

\begin{stmt}
  \label{stmt:center-of-D-standard-is-smooth}
  Let $G$ be a $D$-standard reductive group. Then the center $\zeta_G$
  of $G$ is \emph{smooth.}
\end{stmt}

\begin{proof}
  Indeed, for any field extension $L$ of $K$, the center of $G_{/L}$
  is just the group scheme $(\zeta_G)_{/L}$ obtained by base change.
  To prove that $\zeta_G$ is smooth, it suffices to prove that
  $(\zeta_G)_{/L}$ is smooth. So we may and will suppose that $K$ is
  algebraically closed; in particular, $G$ is split.

  Fix a Borel subgroup $B$ of $G$ and fix a maximal torus $T \subset
  B$.  Let $X = \sum_\alpha X_\alpha \in \Lie(B)$ be the sum over the
  simple roots $\alpha$, where $X_\alpha \in \Lie(B)_\alpha$ is a
  non-zero root vector; then $X$ is \emph{regular nilpotent}.

  For a root $\beta \in X^*(T)$ of $T$ on $\Lie(G)$, write $\beta^\vee
  \in X_*(T)$ for the corresponding cocharacter $\beta^\vee:\Gm \to
  T$, and consider the cocharacter $\phi:\Gm \to T$ given by $\phi =
  \sum_\beta \beta^\vee \in X_*(T)$, where the sum is over all positive roots
  $\beta$. Then $\Ad(\phi(t))X = t^2X$ for each $t \in \Gm(K)$
  so that the image of $\phi$ normalizes the centralizer $C = C_G(X)$.

  Now, $C$ is a smooth subgroup of $G$ by
  \ref{stmt:smooth-centralizers}. The image of $\phi$ is a torus,
  hence is a diagonalizable group. So the fixed points
  $C^{\operatorname{im} \phi}$ of the image of $\phi$ on $C$ form a
  smooth subgroup by \ref{stmt:diag-smooth-fixed-points}.

  Finally, since $X$ is contained in the dense $B$-orbit on
  $\Lie(R_uB)$, $X$ is a \emph{distinguished} nilpotent element; cf.
  \cite{jantzen-nil}*{4.10, 4.13}.  So it follows from
  \cite{jantzen-nil}*{Prop. 5.10}, that $C^{\operatorname{im} \phi}$
  is precisely $\zeta_G$, the center of $G$.  Thus indeed $\zeta_G$ is
  smooth.
\end{proof}

\begin{rem}
  In case $G$ is semisimple in very good characteristic one can
  instead apply \cite{humphreys-conjclasses}*{0.13} to see that the
  center of the Lie algebra $\Lie(G)$ is trivial; this shows in this
  special case that $\zeta_G$ is smooth.
\end{rem}

\subsection{The centralizer of a semisimple element of $\glie$}
\label{sub:centralizer-X}

Suppose $G$ is $D$-standard, let $X \in \glie = \glie(K)$ be
semisimple, and write $M = C_G(X)$. Recall that the closed subgroup
scheme $M$ is smooth over $K$; cf.  \ref{stmt:smooth-centralizers}.

\begin{stmt}
  \label{stmt:C_G(X_s)-reductive}
  \begin{enumerate}
  \item[(a)] $X$ is tangent to a maximal torus $T$ of $G$.
  \item[(b)] $M^o$ is a reductive group.
  \end{enumerate}
\end{stmt}

\begin{proof}
  \cite{borel-LAG}*{Prop. 11.8 and Prop.  13.19}. 
\end{proof}

Now fix a maximal torus $T$ with $X \in \Lie(T)$ as in \ref{stmt:C_G(X_s)-reductive}.
Let us recall the following:
\begin{stmt}
  \label{stmt:levi-as-centralizer}
  If $S \subset G$ is a torus, there is a finite, separable field
  extension $L \supset K$ and a parabolic subgroup $P \subset G_{/L}$
  such that $C_G(S)_{/L}$ is a Levi factor of $P$.
\end{stmt}

\begin{proof}
  Let the finite separable field extension $L \supset K$ be a
  splitting field for $S$. The result then follows from
  \cite{borel-tits}*{4.15}.
\end{proof}

Suppose for the moment that the characteristic $p$ of $K$ is positive.
Let $\Ksep$ be a separable closure of $K$, and consider the (additive)
subgroup $B$ of $\Ksep$ generated by the elements $d\beta(X)$ for
$\beta \in X^*(T_{/\Ksep})$; since $d\beta(X) = 0$ whenever $\beta \in
pX^*(T_{/\Ksep})$, $B$ is a finite elementary Abelian $p$-group.
Write $\Gamma = \Gal(\Ksep/K)$ for the Galois group; since $X \in
\glie(K)$, the group $B$ is stable under the action of $\Gamma$.

Let $\mu = D(B)$ be the $K$-group scheme of multiplicative type determined
by the $\Gamma$-module $B$.  The $\Gamma$-equivariant mapping
$X^*(T_{/\Ksep}) \to B$ given by $\beta \mapsto d\beta(X)$ determines
an embedding of $\mu$ as a closed subgroup scheme of $T$.

\begin{stmt} 
  \label{stmt:C(X)=C(mu)}
  We have $M^o = C_G(\mu)^o$.
\end{stmt}

\begin{proof}[Sketch]
  Since $M^o$ and $C_G(\mu)^o$ are smooth groups over $K$, it suffices
  to give the proof after replacing $K$ by an algebraic closure.  In that case
  $\mu$ is diagonalizable.  Let $R \subset X^*(T)$ be the roots of $G$
  for the torus $T$, and for $\alpha \in R$ let $U_\alpha \subset G$
  be the corresponding root subgroup of $G$.

  Then using the Bruhat decomposition of $G$, one finds that 
  \begin{equation*}
    M^o = \langle T, U_\alpha \mid d\alpha(X) = 0\rangle = C_G(\mu)^o;
  \end{equation*}
  the required argument is essentially the same as that given in
  \cite{springer-steinberg}*{II.4.1} except that \emph{loc. cit.} does not
  treat infinitesimal subgroup schemes; cf. \cite{mcninch-diag} for
  the details.
\end{proof}

\begin{theorem}
  There is a finite separable field extension $L \supset K$ for which
  the connected centralizer $M^o_{/L}=C_G^o(X)_{/L}$ is a Levi factor of a
  parabolic subgroup of $G_{/L}$.
\end{theorem}

\begin{proof}
  Suppose first that $K$ has characteristic $p>0$.  In view of
  \ref{stmt:C(X)=C(mu)}, the reductive group $M^o$ is $D$-standard,
  since $\mu$ is a group of multiplicative type.  According to
  \ref{stmt:center-of-D-standard-is-smooth}, the center $Z$ of $M^o$
  is smooth. Let $S$ be a maximal torus of $Z$.  We have evidently
  $M^o \subset C_G(S)$.  It follows that
  $\Lie(Z) = \Lie(S)$.  We may now use \ref{stmt:Lie(Z)-as-fixed-points}
  to see that $X \in \Lie(Z) = \Lie(S)$. Thus $M^o \supset C_G(S)$.

  It follows that $M^o = C_G(S)$, and we conclude via
  \ref{stmt:levi-as-centralizer}.

  The situation when $K$ has characteristic zero is simpler.  In that
  case, the center $Z$ of the reductive group $M^o$ is automatically
  smooth. If $S$ is a maximal torus of $Z$ then $M^o= C_G(S)$ as
  before.
\end{proof}

\subsection{Borel subalgebras}
\label{sub:borel-subalg}

Suppose that $K$ is algebraically closed. By  a Borel
subalgebra   of $\glie$, we mean
the Lie algebra $\blie = \Lie(B)$ of a Borel subgroup  $B \subset G$.
\begin{prop}[\cite{borel-LAG}*{14.25}]
  $\glie$ is the union of its Borel subalgebras. More precisely, for
  each $X \in \glie$, there is a Borel subalgebra $\blie$ with $X \in
  \blie$.
\end{prop}


\section{The center of a centralizer}
\label{section:center-of-nil-centralizer}

For a $D$-standard reductive group $G$ over $K$, let $x \in G(K)$ and
$X \in \glie(K)$.  We are going to consider the centralizers $C_G(X)$
and $C_G(x)$, and in particular, the centers $Z_x = Z(C_G(x))$ and
$Z_X = Z(C_G(X))$ of these centralizers.  As we have seen, $Z_x$ is a
closed subscheme of $C_G(x)$ and $Z_X$ is a closed subscheme of
$C_G(X)$. In this section, we will prove Theorem
\ref{smoothness-theorem} from the introduction; namely, in
\S\ref{sub:smoothness-of-center-of-centralizer}, we prove that $Z_x$
and $Z_X$ are smooth.  In \S\ref{sub:unipotent case}, we establish
some preliminary results under the assumption that $K$ is perfect.
Since the smoothness of $Z_x$ and of $Z_X$ will follow if it is proved
after base change with an algebraic closure $\Kalg$ of $K$, this
assumption on $K$ is harmless for our needs.

\subsection{Unipotence of the center of the centralizer 
  when $X$ is nilpotent}
\label{sub:unipotent case}

Suppose in this section that the field $K$ is \emph{perfect};
thus if $A$ is a group scheme over $K$, we may speak of the
reduced subgroup scheme $A_\red$ -- cf. \ref{sttm:reduced}.
We begin with the following observation which is due independently
to R. Proud and G. Seitz. For completeness, we include a proof.
\begin{stmt}
  \label{stmt:unipotence-of-center}
  Let $x$ be unipotent, let $X$ be nilpotent, write $C$ for one of
  the groups $C_G(x)$ or $C_G(X)$, and write $Z = Z(C)$; thus
$Z$ is one of the groups $Z_x$ or $Z_X$.
  \begin{enumerate}
  \item[(a)] $C^o$ is not contained in a Levi factor of a proper parabolic
    subgroup of $G$. 
  \item[(b)]   The quotient $(Z_\red)^o/(\zeta_G)^o$ is a unipotent group, where
    $Z_\red$ is the corresponding reduced group, and $(Z_\red)^o$ is its
    identity component.
  \item[(c)] Let $Y \in \Lie(Z)$ be semisimple. Then $Y \in \Lie(\zeta_G)$.
  \end{enumerate}
\end{stmt}

\begin{proof}
  It suffices to prove each of the assertions after extending scalars;
  thus, we may and will suppose in the proof that $K$ is algebraically
  closed.  Moreover, if $\sigma:\NN \to \UU$ is a Springer
  isomorphism, then $C_G(X) = C_G(\sigma(X))$. Thus it suffices to
  give the proof for the centralizer of $X$.

  We first prove (a). Suppose that $L$ is a Levi factor of a parabolic
  subgroup $P$, and assume that $C^o$ is a subgroup scheme of $L$.  Then
  $C^o = C_L^o(X)$ so that $\Lie C = \Lie C_L(X)$. Since $L$ is again a
  $D$-standard reductive group, we see by the smoothness of
  centralizers that $\Lie C_L(X)$ is the centralizer in $\Lie L$ of
  $X$ \ref{stmt:smooth-centralizers}; in particular, it follows that
  every fixed point of $\ad(X)$ on $\Lie(G)$ lies in $\Lie(L)$. If $L$
  were a proper subgroup of $G$, the nilpotent operator $\ad(X)$ would
  have a non-zero fixed point on $\Lie R_uP$; it follows that $L = G$.

  We will now deduce (b) and (c) from (a). For (b), let $S \subset Z$
  be a torus. The assertion (b) will follow if we prove that $S$ is
  central in $G$.  But $L = C_G(S)$ is a Levi factor of some parabolic
  subgroup $P$ of $G$ by \ref{stmt:levi-as-centralizer}, and $C^o
  \subset L$.  Thus by (a) we have $P = G = L$; this shows that $S$ is
  central in $G$, as required.

  For (c), let $Y \in \Lie(Z)$ be semisimple. According to Theorem
  \ref{sub:centralizer-X}, $L = C_G^o(Y)$ is a Levi factor of some
  parabolic subgroup $P$, and $C^o \subset L$. So again (a) shows that
  $P=G = L$. Since $C_G(Y) = G$, it follows that $Y$ is a fixed point
  for the adjoint action of $G$ on $\Lie(G)$.  But according to
  \ref{stmt:Lie(Z)-as-fixed-points}, we have $\Lie(\zeta_G) =
  \Lie(G)^{\Ad(G)}$; thus indeed $Y \in \Lie(\zeta_G)$ as required.
\end{proof}

As a consequence, we deduce the following structural results:
\begin{stmt}
  \label{stmt:Z-structure}
  With notation and assumptions as in \ref{stmt:unipotence-of-center},
  we have:
  \begin{enumerate}
  \item[(a)] $Z_\red$ is the internal direct product $\zeta_G \cdot
    R_u Z_\red$.
  \item[(b)] The set of nilpotent elements of $\Lie(Z)$ forms a
    subalgebra $\ulie$ for which
    \begin{equation*}
      \Lie Z = \Lie(\zeta_G) \oplus      \ulie.
    \end{equation*}
  \end{enumerate}
\end{stmt}

\begin{proof}
  Note that $Z$ and also $\Lie(Z)$ are commutative; since the product
  of two commuting unipotent elements of $G$ is unipotent and the sum
  of two commuting nilpotent elements of $\Lie(G)$ is nilpotent,
  results (a) and (b) follow from \ref{stmt:unipotence-of-center}(b)
  and (c).
\end{proof}

\subsection{Smoothness of the center of the centralizer}
\label{sub:smoothness-of-center-of-centralizer}
In this section, $K$ is again arbitrary.
Let $x \in G(K)$, $X \in \glie(K)$ be arbitrary, write $C$
for one of the groups $C_G(x)$ or $C_G(X)$, and write $Z = Z(C)$,
so that $Z$ is one of the groups $Z_x$ or $Z_X$.
We are now ready to prove the following:
\begin{theorem}
  The center $Z = Z(C)$ is a smooth group scheme over $K$.
\end{theorem}

\begin{proof}
  Since a group scheme is smooth over $K$ if and only if it is smooth
  upon scalar extension, we may and will suppose $K$ to be
  algebraically closed (hence in particular perfect). So as in
  \S\ref{sub:unipotent case}, we may speak of the reduced subgroup
  scheme $A_\red$ of a group scheme $A$ over $K$.

  Let $x = x_sx_u$ and $X = X_s + X_n$ be the Jordan decompositions of
  the elements; thus $x_s \in G$ and $X_s \in \glie$ are semisimple,
  $x_u \in G$ is unipotent, $X_n \in \glie$ is nilpotent, and we have:
  $x_s x_u = x_u x_s$ and $[X_s,X_n] = 0$.

  Then
  \begin{equation*}
    C_G(x) = C_M(x_u) \quad \text{and} \quad C_G(X) = C_M(X_n)
  \end{equation*}
  where $M = C_G(x_s)$ resp. $M = C_G(X_s)$.  

  Now, the Zariski closure of the group generated by $x_s$ is a smooth
  diagonalizable group whose centralizer coincides with $C_G(x_s)$.
  And according to \S \ref{sub:centralizer-X} the centralizer of $X_s$
  is reductive and is the centralizer of a (non-smooth) diagonalizable
  group scheme. Thus in both cases, the connected component of $M$ is
  itself a $D$-standard reductive group. 

  Moreover,   \ref{stmt:diag-centralizer-component-group} shows that $x_u$ is a
  $K$-point of $M^o$.  There is an exact sequence
  \begin{equation*} 
    1 \to C_{M^o}(x_u) \to C_M(x_u) \to E \to 1
  \end{equation*}
  resp.
  \begin{equation*} 
    1 \to C_{M^o}(X_N) \to C_M(X_N) \to E' \to 1
  \end{equation*}
  for a suitable subgroup $E$ resp. $E'$ of $M/M^o$.  Since $M/M^o$
  has order invertible in $K$
  \ref{stmt:diag-centralizer-component-group},
  apply \ref{sub:center-in-disconnected} to see that the smoothness of $Z$
  follows from the smoothness of the center of $C_{M^o}(x_u)$ resp.
  $C_{M^o}(X_n)$; thus the proof of the theorem is reduced to the case
  where $x$ is unipotent and $X$ is nilpotent.  Since in that case
  $C_G(X) = C_G(\sigma(X))$ where $\sigma:\NN \to \UU$ is a Springer
  isomorphism, we only discuss the centralizer of a nilpotent element
  $X \in \glie$.

  We must argue that $\dim Z = \dim \Lie Z$. Since it is a general
  fact that $\dim \Lie Z \ge \dim Z$, it suffices to show the
  following:
  \begin{equation*}
    (*) \quad \dim \Lie Z \le \dim Z.
  \end{equation*}

  By \ref{stmt:Z-structure} we have $\Lie Z = \Lie(\zeta_G) \oplus
  \ulie$ where $\ulie$ is the set of all nilpotent $Y \in \Lie Z$.
  According to \ref{stmt:center-of-D-standard-is-smooth}, the
  center $\zeta_G$ of $G$ is smooth. Thus $\dim \zeta_G = \dim \Lie
  \zeta_G$.  In view of \ref{stmt:Z-structure}, the assertion $(*)$
  will follow if we prove that
  \begin{equation*}
    (**) \quad \dim \ulie \le \dim R_uZ_\red.
  \end{equation*}

  In order to prove $(**)$, we fix a Springer isomorphism $\sigma:\NN
  \to \UU$ -- see Theorem \ref{sub:springer-iso-exists} --, and we
  consider the restriction of $\sigma$ to $\ulie$.
  
  We first argue that $\sigma$ maps $\ulie $ to $R_u Z_\red$.  Since
  $\ulie$ is smooth -- hence reduced -- and since $K$ is algebraically
  closed, it suffices to show that $\sigma$ maps the $K$-points of
  $\ulie$ to $R_u Z_\red$.  Fix $Y \in \ulie(K)$.
  
  If $g \in C_G(X)(K)$, the inner automorphism $\Int(g)$ of $C$ is
  trivial on $Z$; thus, the automorphism $\Ad(g)$ of $\Lie C$ is
  trivial on $\Lie Z$. It follows that
  \begin{equation*}
    g\sigma(Y)g^{-1} = \sigma(\Ad(g)Y) = \sigma(Y).
  \end{equation*}
  Since $K$ is algebraically closed, it now follows from
  \ref{stmt:alg-closed-fixed-points} that 
  \begin{equation*}
    \sigma(Y) \in Z(K) = C_G(X)^{\Int(C_G(X))}(K).
  \end{equation*}
  
  Since $\ulie$ is reduced, one knows $\sigma(Y) \in Z_\red(K)$.
  Since $\sigma(Y)$ is unipotent, it follows that $\sigma(Y) \in
  R_uZ_\red(K)$.

  Thus the restriction of the Springer isomorphism $\sigma$ gives a
  morphism $\sigma_{\mid \ulie}:\ulie \to R_uZ_\red$.  Since $\sigma$
  is a closed morphism, it follows that the image of $\sigma_{\mid
    \ulie}$ is a closed subvariety of $R_uZ_\red$ whose dimension is
  $\dim \ulie$, so that indeed $(**)$ holds.

\end{proof}

With notation as in the preceding proof, we point out a slightly
different argument. Namely, reasoning as above, one can show that the
inverse isomorphism $\tau = \sigma^{-1}:\UU \to \NN$ maps $R_uZ_\red$
to $\ulie$. It follows that $R_uZ_\red$ and $\ulie$ are isomorphic
varieties, hence they have the same dimension.

Note that we have now proved Theorem \ref{smoothness-theorem} from
the introduction.

\section{Regular nilpotent elements}
\label{section:regular}

In this section, we are going to prove Theorems
\ref{N-mod-C-description}, \ref{N-mod-C-weights}, and
\ref{dsigma-multiple-of-identity} from the introduction.  We denote by
$G$ a $D$-standard reductive group over the field $K$. Let $T \subset
G$ be a maximal torus, and let $T_0 \subset T$ where $T_0$ is a
maximal torus of the derived group $G' = (G,G)$ of $G$. Let us write
$r = \dim T_0$ for the semisimple rank of $G$. Finally, let $W =
N_G(T)/T \simeq N_{G'}(T_0)/T_0$ be the corresponding Weyl group.

\subsection{Degrees and exponents}
\label{sub:degrees and exponents}
We give here a quick description of some well-known numerical
invariants associated with the Weyl group $W$. \textsl{We suppose that
  the derived group $G'$ of $G$ is quasi-simple, and we suppose that
  $T$ (and hence $G$) is split over $K$}.

Let $V = X^*(T_0) \tensor_\Z \Q$ and note that the action of the Weyl
group $W$ on $T_0$ determines a linear representation $(\rho,V)$ of
$W$.  The algebra of polynomials (regular functions) on $V$ may be
graded by assigning the degree 1 to each element of the dual space
$V^\vee \subset \Q[V]$. The action via $\rho$ of $W$ on $V$ determines
an action of $W$ on $\Q[V]$ by algebra automorphisms, and it is known that
the algebra $\Q[V]^W$ of $W$-invariant polynomials on $V$ is generated
as a $\Q$-algebra by $r$ algebraically independent
homogeneous elements of positive degree \cite{bourbaki}*{V.5.3 Theorem 3}.
The \emph{degrees} of $W$ are the degrees
$d_1,d_2,\dots,d_r$ of a system of homogeneous generators
for $\Q[V]^W$. The degrees depend
-- up to order -- only on $W$; see \cite{bourbaki}*{V.5.1}.
The \emph{exponents of $W$} are the numbers $k_1,k_2,\dots,k_r$ where
$k_i = d_i - 1$ for $1 \le i \le r$.

Recall that the ``exponents'' earn their name as follows.  Let $c \in
W$ be a Coxeter element \cite{bourbaki}*{V.6.1}, and write $h$ for the
order of $c$.  If $E$ is a field of characteristic 0 containing a
primitive $h$-th root of unity $\varpi \in E^\times$, then
\cite{bourbaki}*{V.6.2 Prop. 3} the eigenvalues of $\rho(c)$ on $V
\tensor_\Q E$ are the values
\begin{equation*}
  \varpi^{k_1},     \varpi^{k_2} ,  \cdots   ,\varpi^{k_r}.
\end{equation*}

The exponents and degrees are known explicitly; cf.
\cite{bourbaki}*{Plate I -- IX}.

\subsection{The centralizer of a regular nilpotent element}
\label{sub:regular-centralizer}
In this section, $G$ is again a $D$-standard reductive group (whose
derived group is not required to be quasisimple) which we assume to be
quasisplit over $K$.

If $\phi:\Gm \to G$ is a cocharacter and $i \in \Z$, we write
$\glie(\phi;i)$ for the $i$-weight space of the action of $\phi(\Gm)$
on $\glie$ under the adjoint action of $\phi(\Gm)$; thus
\begin{equation*}
  \glie(\phi;i) = \{Y \in \glie \mid \Ad(\phi(t))Y = t^i Y \quad
  \forall t \in \Kalg^\times\}.
\end{equation*}
Any cocharacter $\phi$ determines a unique parabolic subgroup
$P=P(\phi)$ whose $\Kalg$ points are given by:
\begin{equation*}
  P(\Kalg) = \{g \in G(\Kalg) \mid \lim_{t \to 0} \Int(\phi(t))g \ \text{exists}\}.
\end{equation*}
One knows that $\plie = \Lie(P) = \sum_{i \ge 0} \glie(\phi;i)$.

Let $X \in \glie(K)$ be nilpotent. Following
\cite{jantzen-nil}*{\S5.3}, we say that a cocharacter $\psi:\Gm \to G$
is said to be \emph{associated to} a nilpotent element $X$ in case (i)
$X \in \glie(\psi;2)$, and (ii) there is a maximal torus $S$ of the
centralizer $C_G(X)$ such that the image of $\psi$ lies in $(L,L)$,
where $L = C_G(S)$.
\begin{stmt}
  \label{stmt:associated-cochar}
  \begin{enumerate}
  \item[(a)] There are cocharacters associated to $X$.
  \item[(b)] If $\phi$ and $\phi'$ are cocharacters associated to $X$,
    then $P(\phi) = P(\phi')$.
  \item[(c)] The centralizer $C_G(X)$ is contained in $P = P(\phi)$
    for a cocharacter $\phi$ associated to $X$.
  \item[(d)] The unipotent radical $R$ of $C_G(X)_{/\Kalg}$ is defined
    over $K$ and is a $K$-split unipotent group.
  \item[(e)]  Any two cocharacters
associated to $X$ are conjugate by a unique element of $R(K)$.
  \end{enumerate}
\end{stmt}

\begin{proof}
  In the geometric setting, these assertions may be found in
  \cite{jantzen-nil}; the existence of an associated cocharacter is an
  essential part of the Bala-Carter, a conceptual proof of which may
  be found \cite{premet}. Over the ground field $K$, (a) and (c)
  follow from \cite{mcninch-rat}*{Theorem 26 and Theorem 28}.  (b)
  follows since associated cocharacters are optimal for the unstable
  vector $X$ in the sense of Kempf; see \cite{premet}.  Finally, (d)
  and (e) follow from \cite{mcninch-optimal}*{Prop/Defn 21}.
\end{proof}

Finally, recall that a nilpotent element $X \in \glie$ is \emph{distinguished}
provided that a maximal torus of the centralizer $C_G(X)$ is central in $G$.

\begin{stmt}
  \label{stmt:characterize-regular-nilpotent} 
  Let $X \in \glie$ be nilpotent. The following are equivalent:
  \begin{enumerate}
  \item[(a)] $X$ is regular -- i.e. $\dim C_G(X)$ is equal to the rank of
    $G$.
  \item[(b)] $X \in \Lie(B)$ for precisely one Borel subgroup of $G$.
  \end{enumerate}
  Moreover, if $X$ is regular then $X$ is distinguished, and if $\phi$ is
  a cocharacter associated with $X$, then $B=P(\phi)$ is the unique
  Borel subgroup with $X \in \Lie(B)$.
\end{stmt}

\begin{proof}
  The equivalence of (a) and (b) can be found in
  \cite{jantzen-nil}*{Cor. 6.8}. Note that in \emph{loc. cit.} it is
  assumed that $K$ is algebraically closed. But, it suffices to prove
  that (b) implies (a) after replacing $K$ by an extension field. It
  remains to argue that (a) implies (b). But given (a), one knows
  there to be a unique Borel subgroup $B \subset G_{/\Kalg}$ with $X
  \in \Lie(B)$, where $\Kalg$ is an algebraic closure of $K$. It now
  follows from \cite{mcninch-rat}*{Prop. 27} that $B$ is a parabolic
  subgroup of $G$ [i.e. that $B$ is defined over $K$], and (b) follows.

  That a regular element is distinguished follows from the Bala-Carter
  theorem; it can be seen perhaps more directly just by observing that
  $B$ is a distinguished parabolic subgroup, so that an elment of the
  dense orbit of $B$ on $\Lie R_uB$ is distinguished by
  \cite{carter}*{5.8.7}.
  
  Finally, write $P = P(\phi)$. It follows from
  \cite{jantzen-nil}*{5.9} that $X$ is in the dense $P$-orbit on
  $\Lie(R_uP)$ and that $C_P(X) = C_G(X)$; thus $\dim \Ad(G)X = 2
  \dim R_uP$ so that indeed $P$ must be a Borel subgroup.
\end{proof}

  


Since $G$ is assumed to be quasisplit, we have
\begin{stmt}[\cite{mcninch-optimal}*{Theorem 54}]
  There is a regular nilpotent element $X \in \glie(K)$.
\end{stmt}
We fix now a regular nilpotent element $X$.   Let $C = C_G(X)$ be
the centralizer of $X$, and write $\zeta_G$ for the center of $G$.

\begin{stmt}
  \label{stmt:centralizer-description} 
  For the group $C = C_G(X)$ we have:
  \begin{itemize}
  \item[(a)] the maximal torus of $C$ is the identity component of
    the center $\zeta_G$ of $G$.
  \item[(b)] $C = \zeta_G \cdot R_u(C)$.
  \item[(c)]   $C$ is commutative.
  \end{itemize}
\end{stmt}

\begin{proof}
  Assertions (a) and (b) follow from \cite{jantzen-nil}*{\S4.10,
    \S4.13} precisely as in the proof of
  \ref{stmt:center-of-D-standard-is-smooth}.

  For (c), use a Springer isomorphism $\sigma:\NN \to \UU$, to see
  that $C$ is the centralizer of the regular unipotent element
  $u=\sigma(X)$.  Then the commutativity of $C$ follows from a result
  of Springer -- see \cite{humphreys-conjclasses}*{Theorem 1.14} --
  which implies that the centralizer of $u$ contains a commutative
  subgroup of dimension equal to the rank of $G$. This shows that the
  identity component of $C$ is commutative. Since $R_uC$ is connected
  and since $C = \zeta_G R_uC$, the group $C$ is itself commutative.
\end{proof}

We now fix a cocharacter $\phi$ of $(G,G)$ associated to $X$.
\begin{stmt}
  \label{stmt:regular-centralizer}
  The image $\phi$ normalizes
  $C$.  Suppose that the derived group of $G$ is quasisimple.  We have
  \begin{enumerate}
  \item[(a)]   \begin{equation*}
    \displaystyle \Lie(R_uC) = \bigoplus_{i = 1}^r
    \Lie(C)(\phi;2k_i)
  \end{equation*}
  where $1=k_1 \le \cdots \le k_r$ are the exponents of the Weyl group
  of $G$.
  \item[(b)] $\dim \Lie(R_uC)(\phi;2) = 1$.
  \end{enumerate}
\end{stmt}

\begin{proof}First suppose that $K$ has characteristic 0. In that
  case, the assertions are a consequence of results of \cite{kostant}.
  One deduces (a) immediately from \cite{kostant}*{\S6.7}.  For (b),
  one knows that the integers $2k_i$ are the highest weights for the
  action of a principal $\lie{sl}_2$ on $\glie$. Examining the roots
  of $\glie$, one knows that the largest weight $2k_r$ occurs
  precisely once; thus $\dim V(\phi;2k_r) = 1$.

  Now the duality of the exponents \cite{kostant}*{Theorem 6.7} shows
  that $$\dim V(\phi;2) = \dim V(\phi;2k_1) = \dim V(\phi;2k_{r}) =
  1$$ as required.

  For general $K$, consider a discrete valuation ring $\A$ whose
  residue field is $K$ and whose field of fractions $L$ has
  characteristic 0, and denote by $\GG$ a split reductive group scheme
  over $\A$ such that upon base change with $K$ one has $\GG_{/K} \simeq
  G$. Of course, the Weyl groups of $\GG_{/K}$ and of $\GG_{/L}$ are isomorphic.

  According to \cite{mcninch-centralizer}*{Theorems 5.4 and 5.7} we
  may find a suitable such $\A$ for which there is a nilpotent section
  $X_0 \in \Lie(\GG)(\A)$ and a homomorphism of $\A$-group schemes
  $\phi:\Gm \to \GG$ with the following properties:
  \begin{enumerate}
  \item[(i)] the image $X_0(K)$ of $X_0$ in $\glie = \Lie(G) =
    \Lie(\GG_{/K})$ coincides with $X$,
  \item[(ii)] the image $X_0(L)$ of $X_0$ in $\Lie(\GG_{/L})$ is
    regular nilpotent,
  \item[(iii)] the cocharacter $\phi_{/K}$ of $G=\GG_{/K}$ is associated
    to $X = X_0(K)$, and
    \item[(iv)] the cocharacter $\phi_{/L}$ of $\GG_{/L}$ is
  associated to $X_0(L)$.
  \end{enumerate}
  Moreover, it follows from \cite{mcninch-centralizer}*{Prop. 5.2}
  that the centralizer subgroup scheme $C_\GG(X_0)$ is smooth. In
  particular, $\Lie(C_\GG(X_0))$ is free as an $\A$-module, and
  $\Lie(C) = \Lie(C_\GG(X_0)) \tensor_\A K$. We may regard
  $\Lie(C_\GG(X_0))$ as a representation for the diagonalizable
  $\A$-group scheme $\Gm$ via $\Ad \circ \phi$.  Decompose this
  representation as a sum of its weight subspaces:
  \begin{equation*}
    \Lie(C_\GG(X_0)) = \bigoplus_{i \in \Z} \Lie(C_\GG(X_0))(\phi;i).
  \end{equation*}
  Extending scalars to $L$, one sees that $\Lie(C_\GG(X_0))(\phi;i)$
  is non-zero if and only if $i/2$ is one of the exponents of the Weyl
  group of $G$, and $\Lie(C_\GG(X_0))(\phi;2)$ has rank 1.
  The assertions (a) and (b) now follow by base change with $K$.
\end{proof}

\subsection{Lifting regular nilpotent elements}



\begin{stmt}
  \label{stmt:center-surj-sep}
  Let $f:G \to H$ be a homomorphism between reductive groups such that
  $f$ is surjective and central -- i.e. the subgroup scheme $\ker f$
  is contained in the center of $G$.  Then $f$ restricts to a
  surjective homomorphism $f_{\mid \zeta_G}:\zeta_G \to \zeta_H$.  
\end{stmt}

\begin{proof}
  The assertion is geometric, so we may and will suppose the field $K$
  to be algebraically closed.  Since $\ker f$ is central, the
  pre-image of each maximal torus $S$ of $H$ is a maximal torus $T$ of
  $G$. Then $f_{\mid T}:T \to S$ is surjective. The result now follows
  because $\zeta_G$ is the (scheme theoretic) intersection of all
  maximal tori in $G$ , and $\zeta_H$ is the intersection of all
  maximal tori in $H$; see \cite{SGA3}*{Exp. XII Prop.  4.10}.
\end{proof}

Suppose now that $G_1$ and $G_2$ are $D$-standard reductive groups,
and that $f:G_1 \to G_2$ is a separable surjective homomorphism of
reductive groups which is central, as before.  Recall that the
separability of $f$ means that the tangent mapping $df$ is surjective.
\begin{stmt}
  \label{stmt:can-lift-regular}
  \begin{enumerate}
  \item[(a)] Suppose that $X_2 \in \Lie(G_2)(K)$ is regular nilpotent.
    There is a nilpotent element $X_1 \in \Lie(G_1)(K)$ for which
    $df(X_1) = X_2$.
  \item[(b)] If $df(Y_1) = Y_2$ for nilpotent elements $Y_i \in \Lie(G_i)$,
    then $Y_1$ is regular if and only if $Y_2$ is regular.
  \end{enumerate}
\end{stmt}

\begin{proof}
  Let $B \subset G_2$ be a Borel subgroup with $X \in \Lie(B)(K)$.  The
  inverse image $B_1$ of $B$ in $G_1$ is a parabolic subgroup
  \cite{borel-LAG}*{22.6}; since $B_1$ is evidently 
  solvable, $B_1$ is a Borel subgroup of $G_1$. Thus $f$ induces a
  morphism $\tilde f:\mathcal{B}_1 = G_1/B_1 \to G_2/B$, and it is
  clear that the tangent map at the point $B_1$ of $\mathcal{B}_1$ is
  an isomorphism.  It follows from \cite{springer-LAG}*{Theorem
    5.3.2(iii)} that $\tilde f$ is an isomorphism between the flag
  varieties.

  Write $\ulie_1 = \Lie R_uB_1$ and $\ulie = \Lie R_uB$.  According to
  \cite{borel-LAG}*{22.5}, $f$ induces a bijection between the roots
  of $G_1$ (with respect to some maximal torus) and the roots of $G$
  (with respect to a compatible maximal torus).  In particular, $\dim
  R_uB_1 = \dim R_u B$. Since $\ker f$ is central in $G$, $\ker df$ is
  contained in $\Lie(T)$ for each maximal torus $T$. It follows that
  the restriction of $df$ to $\ulie_1$ is injective, so that
  $df(\ulie_1) = \ulie$.  Since $X \in \Lie(B)$ is nilpotent, we have
  $X \in \ulie$.  It follows that there is a -- necessarily nilpotent
  -- element $X_1 \in \ulie_1$ with $df(X_1) = X$.  This proves (a).

  Now, $\tilde f$ induces a bijection between the varieties
  $\mathcal{B}_{1,Y_1}$ and $\mathcal{B}_{2,Y_2}$, where
  $\mathcal{B}_{i,Y_i}$ consists of those Borel subgroups $B$ with $Y_i
  \in \Lie(B)$.  Assertion (b) now follows from
  \ref{stmt:characterize-regular-nilpotent}.
\end{proof}

\begin{stmt}
  \label{stmt:central-homomorphism-and-regular-centralizer}
  Suppose that the elements $X_i \in \Lie(G_i)$ are nilpotent for
  $i=1,2$, that $df(X_1) = X_2$, and that $X_1$ is regular,
  equivalently that $X_2$ is regular.  If $C_1 = C_{G_1}(X_1)$ and $C
  = C_{G_2}(X_2)$, then $C_1 = f^{-1}C$. In particular, $f$ restricts to
  a surjective separable mapping $f_{\mid C_1}:C_1 \to C$.
\end{stmt}

\begin{proof}
  As before, the assertion is geometric; thus we may and will suppose
  that $K$ is algebraically closed for the proof.  We only must argue
  that $(*) \quad C_1 = f^{-1}C$.  Indeed, the remaining assertions
  follow from $(*)$ by using \ref{stmt:inverse-image}(d) and the
  smoothness of $C_1$ \ref{stmt:smooth-centralizers}.

  We will argue that $f_{\mid C_1}:C_1 \to C$ is
  surjective; assertion $(*)$ will then follow since $\ker f$ is
  central in $G_1$.  Recall that $C_1 = \zeta_{G_1} \cdot R_uC_1$ and
  $C = \zeta_{G_2} \cdot R_uC$. The restriction $f_{\mid
    \zeta_{G_1}}:\zeta_{G_1} \to \zeta_{G_2}$ is
  surjective \ref{stmt:center-surj-sep}.  

  It remains to argue that $f_{\mid R_uC_1}$ yields a surjective
  mapping $R_uC_1 \to R_uC$.  Since $G_1$ and $G_2$ are $D$-standard,
  the centralizers $C_1$ and $C$ are smooth by
  \ref{stmt:smooth-centralizers}.  Thus the unipotent radicals of
  $C_1$ and of $C$ are smooth group schemes over $K$.  So the
  surjectivity of $f_{\mid R_uC_1}:R_uC_1 \to R_uC$ will follow if we
  only prove that $df:\Lie(R_uC_1) \to \Lie(R_uC)$ is surjective.


  But $df_{\mid \Lie R_uC_1}$ is injective since $\ker df$ is central.
  Moreover, $\dim R_uC_1$ is the semisimple rank of $G_1$, and $\dim
  R_uC$ is the semisimple rank of $G_2$.  Since $f$ is surjective with
  central kernel, the semisimple ranks of $G_1$ and $G_2$ coincide.
  Thus $df_{\mid \Lie R_uC_1}$ is bijective and the assertion follows.
\end{proof}

\subsection{The normalizer of $C$}
\label{sub:normalizer-of-C}
Let us again fix a regular nilpotent element $X$ together with a
cocharacter $\phi$ associated to $X$.  Let $N = N_G(C)$ be the
normalizer of $C$. 

We will argue in \ref{stmt:N-smooth} below that $N$ is a smooth group
scheme over $K$. Meanwhile, we consider in the next assertion the
$N$-orbit of $X$. Viewing this orbit as a subspace of $\Lie(R_uC)$, we
may consider its closure; that closure has a unique structure of
reduced subscheme \cite{liu}*{Prop. 2.4.2}. Since the orbit of $X$ is
open in its closure, that orbit inherits a structure as a reduced
subscheme.

The following argument essentially just records observations made by
Serre in his note found in \cite{mcninch-optimal}*{Appendix}.
\begin{stmt}
  \label{stmt:N-orbit-and-dimension}
  \begin{enumerate}
  \item[(a)]   The $N$-orbit of $X$ is the open subset of $\Lie(R_uC)$
    consisting of the regular elements; i.e.
    \begin{equation*}
      \Ad(N)X = \Lie(R_uC)_\reg
    \end{equation*}
  \item[(b)] The group $N/C$ is connected and has dimension equal to
    the semisimple rank $r$ of $G$.
  \item[(c)] In particular, $\dim N = 2r + \dim \zeta_G$.
  \end{enumerate}
\end{stmt}

\begin{proof}
  Before giving the proof, we recall that $(*)\ C = C^o \cdot \zeta_G$
  where $\zeta_G$ is the center of $G$; see
  \ref{stmt:centralizer-description}.

  For the proof of (a), we have evidently $\Ad(N)X \subset
  \Lie(R_uC)_\reg$. Since $\Ad(N)X$ is a reduced scheme, to prove
  equality it suffices to show that any closed point of
  $\Lie(R_uC)_\reg$ is contained in this orbit. If $\Kalg$ is an
  algebraic closure of $K$ and $Y \in \Lie(R_uC)_\reg(\Kalg)$, then
  $Y$ is a Richardson element for $B$, where $B$ is the Borel subgroup
  as in \ref{stmt:characterize-regular-nilpotent}.   Since the
  Richardson elements form a single orbit under $B$, there is $x \in
  B(\Kalg)$ for which $\Ad(x)Y = X$. Since $C$ is commutative, a
  dimension argument shows that $C_G^o(Y) = C^o$. Since also $C_G(Y) =
  C_G^o(Y) \cdot \zeta_G$; it follows from $(*)$ that $C = C_G(Y)$.
  Since
  \begin{equation*}
    xCx^{-1} = xC_G(Y)x^{-1} = C_G(\Ad(x)Y) = C_G(X) = C, 
  \end{equation*}
  one sees that $x \in N(\Kalg)$. It follows that $\Ad(N)X =
  \Lie(R_uC)_\reg$. 

  For (b), first suppose that $K = \Kalg$ is algebraically closed.  By
  (a), $(N/C)_\red$ identifies with $\Lie(R_uC)_\reg$, an open
  subvariety of the affine space $\Lie(R_uC)$.  It follows that
  $(N/C)_\red$ is an irreducible variety; thus the variety $N/C$ is
  connected.  

  But then relaxing the assumption on $K$, it follows that $N/C$ is
  connected in general.  Since $\Lie(R_uC)$ has dimension equal to
  $r$, conclude that $\dim N/C = r$.

  Finally, (c) follows since $\dim C = r + \dim \zeta_G$.
\end{proof}

We can now prove:

 \begin{stmt}
   \label{stmt:N-smooth}
   $N$ is a smooth subgroup scheme of $G$.
 \end{stmt}

 \begin{proof}
   The statement is geometric; thus we may and will suppose $K$ to be
   algebraically closed. Let $f:G_1 \to G_2$ be a surjective separable
   morphism with central kernel, and suppose that $G$ is one of the
   groups $G_1$ or $G_2$.

   If $G= G_1$, write $X_1$ for $X$ and set $X_2 = df(X_1)$.  If $G =
   G_2$, write $X_2$ for $X$ and use \ref{stmt:can-lift-regular} to
   find a regular nilpotent $X_1 \in \Lie(G_1)$ for which $df(X_1) =
   X_2$.

   Now write $C_i = C_{G_i}(X_i)$.  It follows from
   \ref{stmt:central-homomorphism-and-regular-centralizer} that $C_1 =
   f^{-1}C_2$, so we may apply \ref{stmt:smooth-normalizer} to see
   that 
   \begin{equation*}
     (*) \quad N_{G_1}(C_1) \ \text{is smooth over $K$ if and only if}\
   N_{G_2}(C_2) \ \text{is smooth over $K$}.
   \end{equation*}

   We are now going to argue: it suffices to prove the result
   when $G$ is quasisimple in very good characteristic.  
   
   Well, if the result is known for quasisimple $G$ in very good
   characteristic, it follows easily for any semisimple, simply
   connected group in very good characteristic (since any such is a
   direct product of simply connected quasisimple groups).  But any
   semisimple group in very good characteristic is separably isogenous
   to a simply connected one, so $(*)$ then permits us to deduce the
   result for any semisimple $G$ in very good characteristic.

   For a general $D$-standard group $G$, we must consider a reductive
   group $H$ of the form $H = H_1 \times T$ where $H_1$ is semisimple
   in very good characteristic, together with a diagonalizable
   subgroup scheme $D \subset H$. We suppose that $G$ is separable
   isogenous to $C_H(D)^o$.  The above arguments show that the desired
   result holds for $H$, and we want to deduce the result for $G$.
   Again using $(*)$, we may suppose that $G = C_H(D)^o$.

   But if $N = N_G(C)$, we see that $N = N_H(C_H(X))^D$. Our
   assumption means that $N_H(C_H(X))$ is smooth. But then
   \cite{SGA3}*{Exp.  XI, Cor. 5.3} shows that $N= N_H(C_H(X))^D$ is
   smooth, as required.

   Thus, we now suppose $G$ to be quasisimple in very good
   characteristic.  Now, $\dim N = 2r$ by
   \ref{stmt:N-orbit-and-dimension}, where $r$ is the  rank
   of $G$. Thus to show that $N$ is smooth, we must show that $2r =
   \dim \Lie(N)$.  Since one has always $\dim \Lie(N) \ge \dim N$, it
   is enough to argue that $\dim \Lie(N) \le 2r$.

   Write $\nlie = \{Y \in \glie \mid [Y,\Lie C] \subset \Lie C\}$ for
   the normalizer in $\glie$ of $\Lie(C)$.  Evidently $\Lie(N) \subset
   \nlie$; it therefore suffices to show that $\dim \nlie \le 2r$.
  
   Suppose that $Y \in \nlie$. Since $C$ is commutative, evidently
   $[[Y,X],X] = 0$, so that $Y \in \ker (\ad(X)^2)$. Thus, it suffices to
   show that 
   \begin{equation*}
     (*) \quad \dim \ker (\ad(X)^2) = 2r.
   \end{equation*}

   But in view of our assumptions on the characteristic of $K$, $(*)$
   follows from \cite{springer-IHES}*{Cor. 2.5 and Theorem 2.6}.
 \end{proof}

\begin{stmt} 
  \label{stmt:N-is-solvable}
  $N$ is a solvable group.
\end{stmt}

\begin{proof} 
  Let $B$ be the unique Borel subgroup of $G$ with $X \in \Lie(B)$ as
  in \ref{stmt:characterize-regular-nilpotent}. Since $B$ is solvable,
  the result will follow if we argue that $N \subset B$.

  Since $N$ is smooth -- in particular, reduced -- it suffices to
  argue that $B$ contains each closed point of $N$. Thus, it is enough
  to suppose that $K$ is algebraically closed and prove that $N(K)
  \subset B(K)$.

  Recall first that according to
  \ref{stmt:associated-cochar}(c), we have $C \subset B$.  If $y
  \in N(K)$ it follows that $\Int(y)B$ contains $C$, hence
  $\Lie(\Int(y)B)$ contains $X$.  This proves that $\Int(y)B = B$, so
  $y$ normalizes $B$. Since Borel subgroups are self normalizing, we
  deduce $N(K) \subset B(K)$, and the result follows.
\end{proof}

\begin{stmt}
  \label{stmt:max-torus-of-N}
  Write $S$ for the image of $\phi$ and write $\zeta_G^o$ for the connected
  center of $G$. Then $S \cdot \zeta_G^o$ is a maximal torus of $N$.
\end{stmt}

\begin{proof}
  Let $T \subset N$ be any maximal torus of $N$ containing $S$.  Since
  $T$ commutes with the image of $\phi$, it follows that the space
  $\Lie(C)(\phi;2)$ is stable under $T$. But that space is one
  dimensional \ref{stmt:regular-centralizer} and has $X$ as a basis
  vector; it follows that $X$ is a weight vector for $T$ so that $T$
  lies in the stabilizer in $G$ of the line $[X] \in \PP(\Lie(G))$.
  We know by \ref{stmt:centralizer-description} that $\zeta_G^o$ is a
  maximal torus of $C$; applying \cite{jantzen-nil}*{2.10 Lemma and
    Remark}, one deduces that $S \cdot \zeta_G^o$ is a maximal torus
  of that stabilizer, which completes the proof.
\end{proof}

Note that together \ref{stmt:N-orbit-and-dimension},
\ref{stmt:N-is-solvable}, and \ref{stmt:max-torus-of-N} yield
Theorem \ref{N-mod-C-description} from the introduction.

\begin{stmt}
  \label{stmt:N=SC}
  Consider the line $[X] \in \PP(\Lie(R_uC))$ and let $\OO$ be
  the $N$-orbit of $[X]$.
  \begin{enumerate}
  \item[(a)] The orbit mapping $(a \mapsto [\Ad(a)X]):N \to \mathcal{O}$ is smooth.
  \item[(b)] The stabilizer $\Stab_N([X])$ of $[X]$ in $N$ is smooth
    and is equal to $S \cdot C$.
  \item[(c)] The $N$-orbit  of $[X]$ is open and dense in
    $\PP(\Lie(R_uC))$.
  \end{enumerate}
\end{stmt}

\begin{proof}
  Recall that a mapping $f:X \to Y$ between smooth varieties over $K$
  is smooth if the tangent map $df_x$ is surjective for all closed
  points of $X$. If $X$ and $Y$ are homogeneous spaces for an
  algebraic group, it suffices to check that $df_x$ is surjective for
  one point $x$ of $X$.

  Moreover, it follows from \cite{springer-LAG}*{Prop. 12.1.2} that if
  an algebraic group $H$ acts on a variety $X$, and if $x \in X$ is a
  closed point, then the stabilizer $\Stab_H(x)$ is a smooth subgroup
  scheme if and only if the orbit mapping $H \to H.x$ determined by
  $x$ is a smooth morphism.

  Now, assertion (a) is the content of \cite{mcninch-rat}*{Lemma 23}
  As to (b), first note that the fact that the orbit mapping $N \to
  \OO$ is smooth shows that stabilizer $\Stab_N([X])$ is smooth over
  $K$. Now, according to \cite{jantzen-nil}*{2.10} the stabilizer in
  $G$ of the line $[X]$ is $S \cdot C$. Since $S \cdot C$ is a closed
  subgroup of $N$, the remaining assertion of (b) follows.

  For (c), notice that $\dim N/(S \cdot C) = \dim N/C - 1 = r -1$ by
  \ref{stmt:N-orbit-and-dimension}. Since we have also $\dim
  \PP(\Lie(R_uC)) = r-1$, it follows that the $N$-orbit of $[X]$ is
  open and dense in $\PP(\Lie(R_uC))$, as required.
  \footnote{Alternatively, one can argue as follows. Write $\LL$ for
    the tautological line bundle -- corresponding to the invertible
    sheaf $\mathcal{O}_{\PP(\Lie R_uC)}(-1)$ -- over $\PP(\Lie R_uC)$.
    Then $(\Lie R_uC) \setminus \{0\}$ identifies with the total space
    of $\LL$ with the zero-section removed. It follows that the
    natural mapping $(\Lie R_uC) \setminus \{0\} \to \PP(\Lie R_uC)$
    is flat and hence open.}
\end{proof}

Let us write $D = \Stab_N([X]) = S \cdot C$, and let $\mathbf{1}$ be
the closed point of $N/D$ determined by the trivial coset of $D$ in
$N$. From the adjoint action of the torus $S$ on $\Lie(N)$ one deduces
an action of $S$ on the tangent space $T_\mathbf{1} (N/D)$; thus one
may speak of the weight spaces $T_{\mathbf{1}}(N/D)(\phi;j)$ for $j
\in \Z$.
\begin{stmt}
  \label{stmt:weights-on-N/B}
  Assume that the derived group of $G$ is quasi-simple, and let the
  positive integers $k_1,k_2,\dots,k_r$ be as in \ref{sub:degrees and
    exponents}. Then we have the following:
  \begin{equation*}
    T_\mathbf{1} (N/D) = \bigoplus_{i=2}^r   T_\mathbf{1} (N/D)(\phi;2k_i - 2)
  \end{equation*}
\end{stmt}

\begin{proof}
  Let $\OO \subset \PP(\Lie R_uC)$ be the $N$-orbit of $[X]$.  By
  \ref{stmt:N=SC}(c), one knows that $\OO$ is an open subset of
  $\PP(\Lie(R_uC))$; in particular, $T_{[X]} \OO = T_{[X]}
  \PP(\Lie(R_uC))$. Also by \ref{stmt:N=SC}(c), one knows that the
  orbit mapping $\alpha:N \to \OO$ given by $\alpha(y) = [\Ad(y)X]$
  induces an $S$-equivariant isomorphism $\bar\alpha:N/D \to \OO$.
  Since $\bar\alpha(\mathbf{1}) = [X]$, the tangent map to $\bar
  \alpha$ at $\mathbf{1}$ yields an $S$-isomorphism between
  $T_{\mathbf{1}}(N/D)$ and $T_{[X]} \OO = T_{[X]} \PP(\Lie(R_uC))$.
  The assertion now follows from \ref{stmt:regular-centralizer} and
  the description of the $S$-module structure on the tangent space
  $T_{[X]} \PP(\Lie(R_uC))$ given in
  \ref{stmt:T-weights-on-tangent-to-P(V)}.  .
\end{proof}

We can now complete the proofs of Theorems \ref{N-mod-C-weights} and
\ref{theorem:RuN-split} from the introduction.
\begin{proof}[Proof of Theorem  \ref{N-mod-C-weights}]
  
  Consider the quotient morphism
  \begin{equation*}
    \Phi:N/C \to N/(S \cdot C) = N/D
  \end{equation*}
  and again write $\mathbf{1}$ for the closed point of $N/C$
  determined by the trivial coset, and $\mathbf{1}$ for the closed
  point of $N/D$ determined by the trivial coset.  Then 
  differentiating $\Phi$ gives an $S$-equivariant mapping
  \begin{equation*}
    d\Phi_{\mathbf{1}}:T_{\mathbf{1}}(N/C) \to T_{\mathbf{1}}(N/D).
  \end{equation*}
  Evidently the kernel of $d\Phi_{\mathbf{1}}$ is the image of
  $\Lie(S)$ in $T_{\mathbf{1}}((N/C)$. Regard $T_{\mathbf{1}}(N/C)$ as
  an $S$-module; by complete reducibility one can find an
  $S$-subrepresentation $V \subset T_{\mathbf{1}}(N/C)$ which is a
  complement to $\ker d\Phi_{\mathbf{1}}$.  Then evidently
  $d\Phi_{\mathbf{1}}$ yields an isomorphism between $V$ and
  $T_{\mathbf{1}}(N/D)$, and the assertion of Theorem
  \ref{N-mod-C-weights} follows.
\end{proof}

\begin{proof}[Proof of Theorem \ref{theorem:RuN-split}]
  We must argue that $R_uN$ is defined over $K$ and split.  Keep the
  preceding notations of this section; in particular, $S$ is the image
  of the cocharacter $\phi$ associated to the regular nilpotent
  element $X \in \Lie(G)$. According to Theorem
  \ref{sub:split-uni-radical}, it will suffice to show that $\Lie(S) =
  \Lie(N)^S$ and that each non-0 weight of $S$ on $\Lie(N)$ is
  positive. It suffices to prove these statements after extending
  scalars; thus we may and will suppose that $K$ is algebraically
  closed.

  If $G$ is any $D$-standard reductive group, we may find $D$-standard
  groups $M_1,\dots,M_d$ together with a homomorphism $\Phi:M \to G$
  where $M = \prod_{i=1}^d M_i$, satisfying (a)--(d) of
  \ref{stmt:D-standard-quasisimple-parts}.

  Using \ref{stmt:central-homomorphism-and-regular-centralizer} we may
  find a regular nilpotent element $X_1 \in \Lie(M)$ such that --
  writing $C_1 = C_M(X_1)$ -- the restriction $\Phi_{\mid C_1}:C_1 \to
  C=C_G(X)$ is surjective (and separable).  Moreover, we may choose a
  cocharacter $\phi_1:\Gm \to M$ associated with $X_1$ such that $\phi
  = \Phi \circ \phi_1$ is associated with $X$. Write $S_1 \subset M$
  for the image of $\phi_1$ and $S \subset G$ for the image of $\phi$.

  Now, by \ref{stmt:D-standard-quasisimple-parts}(a) each $M_i$ has
  quasisimple derived group. In the case where $M$ itself has
  quasisimple derived group -- i.e. if $M = M_1$ -- one uses
  \ref{stmt:regular-centralizer} and Theorem \ref{N-mod-C-weights} to
  deduce that 
  \begin{enumerate}
  \item[(i)] $\Lie(S_1) = \Lie(N_1)^{S_1}$, and
  \item[(ii)]  the non-zero
  weights of $S_1$ on $\Lie(N_1)$ are positive,
  \end{enumerate}
  where we have written $N_1 = N_M(C_1)$.
  Since in general $M$ is a direct product of reductive groups each
  having quasisimple derived group, one sees readily that (i) and (ii)
  hold for $M$.

  The normalizer $N_1 = N_M(C_1)$ is smooth by Theorem
  \ref{N-mod-C-description}. Since $\Phi$ is separable, it follows
  from \ref{stmt:smooth-normalizer} that $\Phi_{\mid N_1}:N_1 \to N$
  is surjective and separable -- i.e.  $d\Phi_{\mid N_1}:\Lie(N_1) \to
  \Lie(N)$ is surjective. Using the fact that (i) and (ii) hold
  together with the surjectivity of $d\Phi_{\mid N_1}$, one sees that
  $\Lie(S) = \Lie(N)^S$ and that the non-zero weights of $S$ on
  $\Lie(N)$ are positive, and the proof is complete.
\end{proof}

\subsection{The tangent map to a Springer isomorphism}
In this section, we give the proof of Theorem
\ref{dsigma-multiple-of-identity}.  Thus we suppose in this section
that the derived group of $G$ is quasisimple. We fix a Springer
isomorphism $\sigma:\NN \xrightarrow{\sim} \UU$, and we write $u =
\sigma(X)$ where $u \in G$ is regular unipotent and $X \in \glie$ is
regular nilpotent.

Since $\sigma$ is $G$-equivariant, one knows that $C = C_G(X) = C_G(u)$. 
\begin{stmt}
  The restriction of $\sigma$ to $\Lie R_uC$ determines an isomorphism
  $\gamma:\Lie R_uC \xrightarrow{\sim} R_uC$. In particular, the
  tangent mapping $d\gamma = (d\gamma)_0$ determines an isomorphism
  $d\gamma:\Lie R_uC \xrightarrow{\sim} \Lie R_uC$.
\end{stmt}
\begin{proof}
  Indeed, recall that $C$ is a smooth group scheme, and that $C =
  \zeta_G \cdot R_uC$ by \ref{stmt:centralizer-description}, so that
  $R_uC$ is the space of fixed points of $\Int(u)$ on $\UU$ and $\Lie
  R_uC$ is the space of fixed points of $\Ad(u)$ on $\NN$; the
  assertion is now immediate.
\end{proof}

Write $V = \Lie R_uC$.  Then $d\gamma$ is an endomorphism of $V$ as an
$N$-module, where $N$ is the normalizer in $G$ of $C$.
As in \S\ref{sub:normalizer-of-C}, we fix a cocharacter $\phi$
associated to $X$; write $S \subset N$ for the image of $\phi$.  We
now give the
\begin{proof}[Proof of Theorem \ref{dsigma-multiple-of-identity}.]
  For (1), note first that the mapping $\gamma$ is in particular an
  $S$-module endomorphism of $V$.  Since $\dim V(\phi;2) = 1$ by
  Theorem \ref{stmt:regular-centralizer}, one knows that $X$ spans
  $V(\phi;2)$.  It follows that $d\gamma(X) = \alpha X$ for some
  $\alpha \in K^\times$.

  If now $Y \in V_\reg = (\Lie R_u(C))_\reg$, there is an element $g
  \in N$ with $\Ad(g)X = Y$; cf. \ref{stmt:N-orbit-and-dimension}. Then
  \begin{equation*}
    d\gamma(Y) = d\gamma(\Ad(g)X) = \Ad(g)d\gamma(X) = 
    \alpha \Ad(g)X = \alpha Y.
  \end{equation*}
  It follows that  $d\gamma$ and $\alpha \cdot 1_V$ agree
  on the dense subset $(\Lie(R_uC))_\reg \subset \Lie(R_uC)$ so that
  indeed $d\gamma = \alpha \cdot 1_V$.

  For (2), recall that $B$ is a Borel subgroup of $G$ with unipotent
  radical $U$.  That $\sigma_{\mid \Lie U}$ is an isomorphism onto $U$
  follows from \cite{mcninch-optimal}*{Remark 10}.

  Now fix a Richardson element $X \in \Lie(U)(K)$; then $X$ is a
  regular nilpotent element of $\glie$, and part (1) shows that
  $d\sigma_{\mid \Lie U}(X) = \alpha X$ for some $\alpha \in
  K^\times$.  If $Y \in \Lie(U)(\Kalg)$ is a second Richardson
  element, then $Y = \Ad(g)X$ for $g \in B(\Kalg)$, and it is then
  clear by the equivariance of $d(\sigma_{\mid \Lie U})_0$ that
  $d(\sigma_{\mid \Lie U})_0(Y) = \alpha Y$. Since the Richardson
  elements are dense in $\Lie U$, the result follows.
\end{proof}

Note that Theorem \ref{dsigma-multiple-of-identity} need not hold when
the derived group of $G$ fails to be quasi-simple. Indeed, take for
$G$ the $D$-standard group $G = \GL_n \times \GL_m$ where $n,m \ge 2$.
Then $\glie = \lie{gl}_n \oplus \lie{gl}_m$, and the mapping
\begin{equation*}
  (X,Y) \mapsto (1 + \alpha X, 1 + \beta Y)
\end{equation*}
defines a Springer isomorphism $\sigma$ for any $\alpha,\beta \in
K^\times$.  If $X_0 \in \lie{gl}_n$ and $Y_0 \in \lie{gl}_m$ are
regular nilpotent, then $X = (X_0,Y_0) \in \glie$ is regular
nilpotent; the mapping $d\sigma$ has eigenvalues $\alpha$ and $\beta$
on $\Lie R_uC_G(X)$ and hence is not a multiple of the identity if
$\alpha \ne \beta$.



We finally conclude with an argument which gives an alternate proof of
(b) of Theorem \ref{smoothness-theorem} in case $G$ has quasi-simple
derived group. This argument does not rely on the fact that $Z(C_1)$
is smooth; on the other hand, in order to make sense of
$Z(C_1)_\red$, we are forced to assume $K$ to be perfect.

\begin{stmt}
  Let $K$ be perfect, let $X_1 \in \glie(K)$ be nilpotent, and let $C_1 =
  C_G(X_1)$ be its centralizer. Then the rule $t \mapsto \sigma(tX_1)$
  defines a mapping $\Phi:\Aff^1 \to Z(C_1)_\red$, and $X_1 = c \cdot
  d\Phi_0(1) \in \Lie(Z(C_1)_\red)$ for some $c \in K^\times$.
\end{stmt}

\begin{proof}
  Let $u = \sigma(X_1)$ and observe that $C_1 = C_G(u)$ by the
  $G$-equivariance of $\sigma$, so in particular, $u \in C_1$. Then
  for each $t \in \Aff^1$, and for each $g \in C_1$, we have
  \begin{equation*}
    g \cdot \sigma(tX_1)\cdot  g^{-1}  = \sigma(t \Ad(g)X_1 ) = \sigma(tX_1).
  \end{equation*}
  Since $\Aff^1$ is reduced, it follows that $\sigma(tX_1)$ indeed
  lies in $Z(C_1)_\red$.
 
  The formula for the tangent mapping of $\Phi$ is now immediate from
  Theorem \ref{dsigma-multiple-of-identity}.
\end{proof}


\newcommand\mylabel[1]{#1\hfil}


\begin{bibsection}

  \begin{biblist}[\renewcommand{\makelabel}{\mylabel} \resetbiblist{XXXXX}]

  \bib{borel-LAG}{book}{
      author = {Borel, Armand},
      title = {Linear Algebraic Groups},
      year = {1991},
      publisher = {Springer Verlag},
      volume = {126},
      series = {Grad. Texts in Math.},
      edition = {2nd ed.},
      label = {Bor 91}}

    \bib{borel-tits}{article}{
      author = {Borel, Armand},
      author = {Tits, Jacques},
      title = {Groupes R\'eductifs},
      journal = {Publ. Math. IHES},
      volume = {27},
      year = {1965},
      pages = {55--151},
      label = {BoT 65}}

    \bib{bourbaki}{book}{
      author={Bourbaki, Nicolas},
      title={Lie groups and Lie algebras. Chapters 4--6},
      series={Elements of Mathematics (Berlin)},
      note={Translated from the 1968 French original by Andrew Pressley},
      publisher={Springer-Verlag},
      place={Berlin},
      date={2002},
      label = {Bou 02}}

  \bib{carter}{book}{
    author = {Carter, Roger W.},
    title = {Finite groups of {L}ie type},
    publisher = {John Wiley \& Sons, Ltd},
    place = {Chichester},
    year={1993},
    note = {Reprint of the 1985 original},
    label = {Ca 93}}


    \bib{DG}{book}{
      author = {Demazure, M.},
      author = {Gabriel, P.},
      title = {Groupes Alg\'ebriques},
      publisher = {Masson/North-Holland},
      place = {Paris/Amsterdam},
      year = {1970}      ,
      label = {DG 70}      }


       \bib{SGA3}{book}{
       author = {Demazure, M.},
       author = {Grothendieck, A.},
       title  = {Sch\'emas en Groupes (SGA 3)},
       series = {S\'eminaire de G\'eometrie Alg\'ebrique du Bois Marie},
       year  = {1965},
       label = {SGA3}      }


      
    \bib{Humphreys-Coxeter}{book}{
      author = {Humphreys, James E.},
      title = {Reflection Groups and Coxeter Groups},
      series = {Cambridge Studies in Advanced Math.},
      publisher = {Cambridge Univ Press},
      volume = {29},
      year = {1990},
      label = {Hum 90}
      }

  \bib{humphreys-conjclasses}{book}{ 
    author={Humphreys, James~E.},
    title={Conjugacy classes in semisimple algebraic groups},
    series={Math. Surveys and Monographs}, 
    publisher={Amer. Math. Soc.}, 
    date={1995}, 
    volume={43},
    label={Hum 95}}

    \bib{JRAG}{book}{
      author={Jantzen, Jens Carsten},
      title={Representations of algebraic groups},
      series={Mathematical Surveys and Monographs},
      volume={107},
      edition={2},
      publisher={American Mathematical Society},
      place={Providence, RI},
      date={2003},
      pages={xiv+576},
      label = {Ja 03}
    }

    \bib{jantzen-nil}{incollection}{ 
      author={Jantzen, Jens~Carsten},
      booktitle = {Lie Theory: Lie Algebras and Representations},
      series = {Progress in Mathematics},
      publisher = {Birkh\"auser},
      editor = {Anker, J-P},
      editor = {Orsted, B},
      place = {Boston},
      volume = {228},
      title={Nilpotent orbits in representation theory}, 
      date = {2004},
      pages = {1\ndash211},
      label = {Ja 04}}

  \bib{kempf-instab}{article}{
    author={Kempf, George~R.},
     title={Instability in invariant theory},
      date={1978},
      ISSN={0003-486X},
   journal={Ann. of Math. (2)},
    volume={108},
    number={2},
     pages={299\ndash 316},
    label ={Ke 78} }      


      \bib{kostant}{article}{
        author={Kostant, Bertram},
        title={The principal three-dimensional subgroup and the Betti numbers of
          a complex simple Lie group},
        journal={Amer. J. Math.},
        volume={81},
        date={1959},
        pages={973--1032},
        label = {Ko 59}      }
		

    \bib{KMRT}{book}{ 
      author={Knus, Max-Albert}, 
      author={Merkurjev,  Alexander}, 
      author={Rost, Markus}, 
      author={Tignol, Jean-Pierre},
      title={The book of involutions}, 
      series={Amer. Math. Soc. Colloq.  Publ.}, 
      publisher={Amer. Math. Soc.}, 
      date={1998}, 
      volume={44},
      label = {KMRT}  }

    \bib{liu}{book}{
      author = {Liu, Qing},
      title = {Algebraic geometry and arithmetic curves},
      note = {Translated from the French by Reinie Ern\'e},
      series={Oxford Graduate Texts in Mathematics},
      number = {6},
      publisher={Oxford University Press},
      year = {2002},
      label = {Li 02}}


    \bib{mcninch-testerman}{article}{
      author={McNinch, George~J.},
      author={Testerman, Donna~M.},
      title={Completely reducible $\operatorname{SL}(2)$-homomorphisms},
      journal = {Transact. AMS},
      volume = {359}
      year = {2008},
      pages = {4489\ndash 4510}
      label={MT 07}}

    \bib{mcninch-rat}{article}{
      author = {McNinch, George~J.},
      title = {Nilpotent orbits over ground fields of good characteristic},
      volume = {329},
      pages = {49\ndash 85},
      year = {2004},
      note={arXiv:math.RT/0209151},
      journal = {Math. Annalen},
      label={Mc 04}}

    \bib{mcninch-optimal}{article}{
      author={McNinch, George~J.},
      title = {Optimal $\operatorname{SL}(2)$-homomorphisms},
      date = {2005},
      volume = {80},
      pages = {391 \ndash 426},
      journal = {Comment. Math. Helv.},
      label = {Mc 05}}

    \bib{mcninch-centralizer}{article}{
      author  = {McNinch, George~J.},
      title = {The centralizer of a nilpotent section},
      journal = {to appear Nagoya Math J.},
      note  = {arXiv:math/0605626},
      label = {Mc 08}
    }

    \bib{mcninch-diag}{article}{
      author = {McNinch, George~J.},
      title = {The centralizer of a subgroup of multiplicative type},
      note = {in preparation},
      label = {Mc 08a}
      }
    
     \bib{premet}{article}{
       author={Premet, Alexander},
       title={Nilpotent orbits in good characteristic and the Kempf-Rousseau theory},
       note={Special issue celebrating the 80th birthday of Robert
         Steinberg.},
       journal={J. Algebra},
       volume={260},
       date={2003},
       number={1},
       pages={338\ndash 366},
       label = {Pr 03}}

    \bib{springer-LAG}{book}{ 
      author={Springer, Tonny~A.}, 
      title={Linear algebraic groups}, 
      edition={2}, 
      series={Progr. in Math.},
      publisher={Birkh{\"a}user}, address={Boston}, date={1998},
      volume={9}, 
      label={Sp 98}}

    \bib{springer-IHES}{article}{
      author={Springer, T. A.},
      title={Some arithmetical results on semi-simple Lie algebras},
      journal={Inst. Hautes \'Etudes Sci. Publ. Math.},
      number={30},
      date={1966},
      pages={115--141},
      label = {Spr 66}
    }
		

    
    \bib{springer-iso}{incollection}{
      author     ={Springer, Tonny~A.},
      title      ={The unipotent variety of a semi-simple group},
      date       ={1969},
      booktitle  ={Algebraic geometry (Internat. Colloq., Tata Inst. Fund. Res.,
        Bombay, 1968)},
      publisher  ={Oxford Univ. Press},
      address    ={London},
      pages      ={373\ndash 391},
      label      ={Spr69}}


    \bib{springer-steinberg}{incollection}{ 
      author={Springer, Tonny~A.},
      author={Steinberg, Robert}, 
      title={Conjugacy classes},
      date={1970}, 
      booktitle={Seminar on algebraic groups and related
        finite groups (The Institute for Advanced Study, Princeton,
        N.J., 1968/69)}, 
      publisher={Springer}, 
      address={Berlin},
      pages={167\ndash 266}, 
      note={Lecture Notes in Mathematics, Vol.  131}, 
      review={\MR{42 \#3091}}, 
      label = {SS 70}}
  

    \bib{steinberg-endomorphisms}{book}{
      author={Steinberg, Robert},
      title={Endomorphisms of linear algebraic groups},
      series={Memoirs of the American Mathematical Society, No. 80},
      publisher={American Mathematical Society},
      place={Providence, R.I.},
      date={1968},
      pages={108},
      label = {St 68}}




\end{biblist}
\end{bibsection}
\end{document}